\documentclass[10pt,a4paper,twoside,reqno]{amsart}

\usepackage{xcolor,soul,lipsum}
\usepackage{hyperref,url}
\usepackage{amssymb,latexsym}
\usepackage{amsmath,amsthm}
\usepackage{enumerate}
\usepackage[all]{xy} \CompileMatrices
\usepackage{amscd}
\hypersetup{colorlinks=false,pdfborderstyle={/S/U/W 0}}
\usepackage{slashed} %% for the Dirac operator
\usepackage{tikz}
\usepackage{float}
\usetikzlibrary{matrix}
\usepackage[normalem]{ulem}
\SelectTips{cm}{12} % computer modern fonts of size 12 for arrowheads
\theoremstyle{plain}
\newtheorem{theorem}{Theorem}[section]
\newtheorem{lemma}[theorem]{Lemma}
\newtheorem{corollary}[theorem]{Corollary}
\newtheorem{proposition}[theorem]{Proposition}

\theoremstyle{definition}
\newtheorem{definition}[theorem]{Definition}
\newtheorem{definition-theorem}[theorem]{Definition-Theorem}
\newtheorem{example}[theorem]{Example}
\theoremstyle{remark}
\newtheorem{remark}[theorem]{Remark}

\numberwithin{equation}{section} 

\usepackage{geometry}
\usepackage{enumitem}

\usepackage[utf8]{inputenc}
\usepackage{color}

\newcommand{\End}{\mathrm{End}}

\newcommand{\Ker}{\mathrm{Ker}}

\newcommand{\surj}{\to\kern-1.8ex\to}

\newcommand{\cA}{\mathcal{A}}

\newcommand{\Diff}{\mathrm{Diff}}
\newcommand{\cC}{\mathcal{C}}
\newcommand{\Conf}{\mathrm{Conf}}
\newcommand{\Sol}{\mathrm{Sol}}
\newcommand{\Met}{\mathrm{Met}}

\newcommand{\cE}{\mathcal{E}}
\newcommand{\cO}{\mathcal{O}}

\newcommand{\arxiv}[1]{{\tt
		\href{http://www.arXiv.org/abs/#1}{arXiv:#1}}}

\newcommand{\cS}{\mathcal{S}}
\newcommand{\cI}{\mathcal{I}}
\newcommand{\dd}{\mathrm{d}}

\newcommand{\mR}{\mathcal{R}}

\setlist[itemize]{leftmargin=*}

\begin{document}

\title[The Heterotic-Ricci flow and its three-dimensional solitons]{The Heterotic-Ricci flow and its three-dimensional solitons}

\author[Andrei Moroianu]{Andrei Moroianu}
\address{Universit\'e Paris-Saclay, CNRS,  Laboratoire de math\'ematiques d'Orsay, France}
\email{andrei.moroianu@math.cnrs.fr}

\author[\'Angel J. Murcia]{\'Angel J. Murcia}
\address{Istituto Nazionale di Fisica Nucleare, Sezione di Padova, Repubblica Italiana}
\email{angel.murcia@pd.infn.it}

\author[C. S. Shahbazi]{C. S. Shahbazi} \address{Departamento de Matem\'aticas, Universidad UNED - Madrid, Reino de Espa\~na}
\email{cshahbazi@mat.uned.es} 
\address{Fachbereich Mathematik, Universit\"at Hamburg, Deutschland}
\email{carlos.shahbazi@uni-hamburg.de}

\thanks{\'A. J. M. has been supported by a postdoctoral fellowship from the Istituto Nazionale di Fisica Nucleare, Bando 23590. The work of C. S. S. was partially supported by the research Excellency Mar\'ia Zambrano grant and the Leonardo grant \emph{LEO22-2-2155} of the BBVA. A. M. and C. S. S. are grateful to the Mathematisches Forschungsinstitut Oberwolfach for support within the OWRF program. C. S. S. is also grateful to the Simons Center for Geometry and Physics for its hospitality during the later stages of this project and to the organizers and participants of the workshop \emph{Supergravity, Generalized Geometry, and Ricci Flow} for related stimulating discussions and comments on this article.}

\maketitle

\begin{abstract}
We introduce a novel curvature flow, the Heterotic-Ricci flow, as the two-loop renormalization group flow of the Heterotic string common sector and study its three-dimensional compact solitons. The Heterotic-Ricci flow is a coupled curvature evolution flow, depending on a non-negative real parameter $\kappa$, for a complete Riemannian metric and a three-form $H$ on a manifold $M$. Its most salient feature is that it involves several terms quadratic in the curvature tensor of a metric connection with skew-symmetric torsion $H$. When $\kappa = 0$ the Heterotic-Ricci flow reduces to the generalized Ricci flow and hence it can be understood as a modification of the latter via the second-order correction prescribed by Heterotic string theory, whereas when $H=0$ and $\kappa >0$ the Heterotic-Ricci flow reduces to a constrained version of  the RG-2 flow and hence it can be understood as a generalization of the latter via the introduction of the three-form $H$. Solutions of Heterotic supergravity with trivial gauge bundle, which we call \emph{Heterotic solitons}, define a particular class of three-dimensional solitons for the Heterotic-Ricci flow and constitute our main object of study. We prove a number of structural results for three-dimensional Heterotic solitons, obtaining, in particular, the complete classification of compact three-dimensional \emph{strong} Heterotic solitons as hyperbolic three-manifolds or quotients of the Heisenberg group equipped with a left-invariant metric. Furthermore, we prove that all Einstein three-dimensional Heterotic solitons have constant dilaton and leave as open the construction of a Heterotic soliton with non-constant dilaton. In this direction, we prove that Einstein Heterotic solitons with constant dilaton are rigid and therefore cannot be deformed into a solution with non-constant dilaton. This is, to the best of our knowledge, the first rigidity result for compact supergravity solutions in the literature.
\end{abstract}

\setcounter{tocdepth}{1} %doesn't display subsections in TOC 
\tableofcontents

%%%%%%%%%%%%%%%%%%%%%%%%%%%%%%%%%%%%%%%%%%%%%%%%%%%%%%%%%%%%%%%%%%%%%%%%%%%%
%%%%%%%%%%%%%%%%%%%%%%%%%%%%%%%%%%%%%%%%%%%%%%%%%%%%%%%%%%%%%%%%%%%%%%%%%%%%
%%%%%%%%%%%%%%%%%%%%%%%%%%%%%%%%%%%%%%%%%%%%%%%%%%%%%%%%%%%%%%%%%%%%%%%%%%%%
%%%%%%%%%%%%%%%%%%%%%%%%%%%%%%%%%%%%%%%%%%%%%%%%%%%%%%%%%%%%%%%%%%%%%%%%%%%%

\section{Introduction}
\label{sec:intro}

%%%%%%%%%%%%%%%%%%%%%%%%%%%%%%%%%%%%%%%%%%%%%%%%%%%%%%%%%%%%%%%%%%%%%%%%%%%%
%%%%%%%%%%%%%%%%%%%%%%%%%%%%%%%%%%%%%%%%%%%%%%%%%%%%%%%%%%%%%%%%%%%%%%%%%%%%
%%%%%%%%%%%%%%%%%%%%%%%%%%%%%%%%%%%%%%%%%%%%%%%%%%%%%%%%%%%%%%%%%%%%%%%%%%%%
%%%%%%%%%%%%%%%%%%%%%%%%%%%%%%%%%%%%%%%%%%%%%%%%%%%%%%%%%%%%%%%%%%%%%%%%%%%%

%%%%%%%%%%%%%%%%%%%%%%%%%%%%%%%%%%%%%%%%%%%%%%%%%%%%%%%%%%%%%%%%%%%%%%%%%%%%

\subsection{Background and motivation}

%%%%%%%%%%%%%%%%%%%%%%%%%%%%%%%%%%%%%%%%%%%%%%%%%%%%%%%%%%%%%%%%%%%%%%%%%%%%

The main purpose of this article is to introduce a novel curvature flow, namely the Heterotic-Ricci flow, and prove a number of structural results about the particular class of its three-dimensional solitons that can be interpreted as solutions of Heterotic supergravity \cite{BRI,BRII}. The Heterotic-Ricci flow and its associated solitons provide a general framework in which a number of notable curvature flows as well as their solitons, all inspired to some extent by some of the building blocks of Heterotic string theory, appear as natural particular cases, hence giving a conceptual understanding of these flows as limiting cases of a well-defined geometric construction determined by \emph{Heterotic supersymmetry}. The Heterotic-Ricci flow, which we introduce in Definition \ref{def:Heterotic} as the two-loop renormalization group flow of the Heterotic string common sector, involves a real parameter $\kappa$ and evolves a family of metrics $g_t$ and three-forms $H_t$ coupled via a system of equations that contains several quadratic terms in the curvature tensor of the unique metric connection $\nabla^{g_t,H_t}$ with skew-torsion $-H_t$. In particular, the Heterotic-Ricci flow is related to other recent or more classical flows as follows \cite{Phong:2018wgn,PhongReview}:

\begin{itemize}
    \item When $\kappa =0$ and $H_t = 0$, the Heterotic-Ricci flow is the classical Ricci-flow \cite{Hamilton}.
    \item When $H_t=0$, the Heterotic-Ricci flow reduces to the RG 2-flow \cite{GGI} involving a quadratic constraint on the Riemann tensor that is inherited from the Bianchi identity of the system.  To the best of our knowledge this constraint has not been studied in the RG 2-flow literature, and, based on previous experience with anomaly cancellation constraints, it may lead to a better-behaved RG 2-flow.
    \item When $\kappa = 0$, the Heterotic-Ricci flow reduces to the generalized Ricci flow \cite{FernandezStreetsLibro,Oliynyk,Streets0}.
    \item If we consider the Heterotic-Ricci flow on a complex manifold with the trivial canonical bundle and if, using the standard notation in the literature, we set $H_t = -\dd^c \omega_t$  for a family of Hermitian metrics $\omega_t$, then the Bianchi identity of the Heterotic-Ricci flow corresponds to the main equation of the Anomaly flow \cite{Phong:2018wgn,Anomaly,AnomalyII}. 
    \item In the set-up of the previous bullet point, if in addition, we assume $\kappa = 0$ then we obtain the generalized Kähler-Ricci flow \cite{Streets}, which in turn is naturally related to the pluriclosed flow \cite{StreetsTian,StreetsTianII} as explained in \cite{Garcia-FernandezJordanStreets,StreetsTianIII}.
    \item The Heterotic-Ricci flow should be a particular case of the generalized Ricci flow on a string Courant algebroid as introduced in \cite{Garcia-Fernandez:2016ofz}, although this remains to be checked explicitly. Clarifying the precise condition on the underlying Courant algebroid that reduces the flow of \cite{Garcia-Fernandez:2016ofz} to the Heterotic-Ricci flow would clarify the behavior of the latter under T-duality, see also \cite{Severa:2018pag,StreetsT,Hassle}.
    \item Being the renormalization group flow of a string theory, the Heterotic-Ricci flow can be expected to be related to other curvature evolution flows inspired by string theory \cite{Fei:2018mzf,Fei:2020kkl,Fei:2020nwv}, although the precise relationship may not be straightforward and may require non-trivial string theory dualities.
\end{itemize}

\noindent
Naturally, following the previous list, the solitons of the Heterotic-Ricci flow are similarly related to the solitons of the Ricci flow \cite{Cao}, RG-2 flow \cite{GarciaMarinoVazquez,GlickensteinWu}, generalized Ricci flow \cite{FernandezStreetsLibro,StreetsSoliton,StreetsUstinovskiySoliton}, and Anomaly flow \cite{Phong:2018wgn,Anomaly,AnomalyII}, respectively. In particular,  solutions to the celebrated Hull-Strominger system \cite{Fei,FeiYau,FernandezIvanovUgarteVillacampa,FernandezIvanovUgarteVassilev,FuTsengYau,FuYau,Garcia-Fernandez:2018emx,Hull,Strominger} or the coupled Hermite-Einstein equations, recently introduced in  \cite{MarioRaul,MarioRaulII}, are automatically particular solutions of the Einstein equation of Heterotic supergravity and hence are intimately related to the solitons of the Heterotic-Ricci flow. Note that, whereas the Hull-Strominger system involves a single higher curvature term, which appears in its Bianchi identity, the soliton system of the Heterotic-Ricci flow will involve three different higher curvature terms, all of them quadratic in the curvature of a metric connection with skew torsion.

The Heterotic Ricci flow, although neatly packed in terms of curvature tensors with torsion, is a very complex curvature flow. As a first attempt to understand its main properties, we focus in this article on the three-dimensional solitons of the system, and in particular on those admitting an interpretation as solutions of three-dimensional Heterotic supergravity. This leads us to define the \emph{Heterotic system} in Section \ref{sec:HS} together with its strong version, which is obtained by exploiting the internal consistency of the system and which admits a natural higher-dimensional geometric interpretation in terms of the total space of the underlying frame bundle.

%%%%%%%%%%%%%%%%%%%%%%%%%%%%%%%%%%%%%%%%%%%%%%%%%%%%%%%%%%%%%%%%%%%%%%%%%%%%

\subsection{Main results and overview}

%%%%%%%%%%%%%%%%%%%%%%%%%%%%%%%%%%%%%%%%%%%%%%%%%%%%%%%%%%%%%%%%%%%%%%%%%%%%

\begin{itemize}
    \item In Section \ref{sec:Heteroticflow} we propose and define the Heterotic-Ricci flow as the two-loop renormalization group flow of the Heterotic string common sector. We reformulate it in three dimensions and we study a particular class of Heterotic-Ricci flows, that we call \emph{homothety flows}, for which the metric evolves only by homotheties. We find the explicit homothety flows and we perform a detailed numerical analysis in the most difficult cases where explicit expressions cannot be obtained. In the latter cases, we obtain explicit numerical examples of eternal and regular flows, showcasing that these flows can be better behaved than their counterparts in the generalized Ricci flow or the RG-2 flow.  

    \item In Section \ref{sec:HS} we introduce the Heterotic soliton system as the equations of motion of Heterotic supergravity with the trivial gauge bundle. Studying the consistency of the system we arrive at the notion of \emph{strong Heterotic soliton}, which we interpret geometrically in terms of the total space of the frame bundle of the underlying Riemannian manifold.

    \item In Section \ref{sec:solitonsHR} we study compact three-dimensional Heterotic solitons. In Theorem \ref{thm:constantdilatonEinstein} we prove that all Einstein three-dimensional Heterotic solitons have constant dilaton. Furthermore,  in Theorem \ref{thm:strongHeterotic3d}  we prove that all three-dimensional compact strong Heterotic solitons have constant dilaton and are either hyperbolic manifolds or quotients of the Heisenberg group.

    \item In Section \ref{sec:rigidity} we study the local deformations of Einstein Heterotic solitons. In particular, in Theorem \ref{thm:rigidity} we prove that all three-dimensional compact Einstein Heterotic solitons are rigid and therefore cannot be deformed into Heterotic solitons with non-constant dilaton. To the best of our knowledge, this seems to be the first rigidity result for a compact solution of a supergravity theory. 
\end{itemize}

\noindent
It should be mentioned that being a new curvature flow, the most fundamental properties of the Heterotic-Ricci flow remain to be investigated, especially its weak parabolicity regime and short-time existence. The latter has been studied in \cite{Buckland,GimreGuentherIsenberg} for the case of the RG 2-flow, proving the short-time existence of the flow assuming a bound on the sectional curvature of the initial data. Similarly, the basic properties of the solitons of the Heterotic-Ricci flow remain to be studied, especially the construction of examples, their higher-categorical interpretation, and the study of the associated moduli and stability. In particular, it would be interesting to extend to the Heterotic-Ricci case the work of \cite{GarciaMarinoVazquez,GlickensteinWu} for solitons of the RG 2-flow. Particularly intriguing is the study of the solitons of the Heterotic-Ricci flow on a complex manifold, with an adapted Ansatz as proposed in \cite{Garcia-FernandezJordanStreets,StreetsSoliton,StreetsUstinovskiySoliton,StreetsUstinovskiySolitonII} for the solitons of the generalized Ricci flow, and their relation with the coupled Hermite-Einstein system of \cite{MarioRaul,MarioRaulII}. Related to the existence of solitons for the flow, we have not been able to construct a three-dimensional Heterotic soliton with non-constant dilaton. Regarding this problem, we do not have enough evidence to make a conjecture in either sense, so constructing an example with non-trivial dilaton or proving that its non-existence constitutes the main open problem that can be derived from this article.

%%%%%%%%%%%%%%%%%%%%%%%%%%%%%%%%%%%%%%%%%%%%%%%%%%%%%%%%%%%%%%%%%%%%%%%%%%%%
%%%%%%%%%%%%%%%%%%%%%%%%%%%%%%%%%%%%%%%%%%%%%%%%%%%%%%%%%%%%%%%%%%%%%%%%%%%%

\section{The Heterotic-Ricci flow}
\label{sec:Heteroticflow}

%%%%%%%%%%%%%%%%%%%%%%%%%%%%%%%%%%%%%%%%%%%%%%%%%%%%%%%%%%%%%%%%%%%%%%%%%%%%
%%%%%%%%%%%%%%%%%%%%%%%%%%%%%%%%%%%%%%%%%%%%%%%%%%%%%%%%%%%%%%%%%%%%%%%%%%%%

In this section, we introduce the \emph{Heterotic-Ricci flow} as the renormalization group flow of the Neveu-Schwarz sector of the Heterotic string at second order in the string-slope parameter $\kappa$. At first order in $\kappa$ this renormalization group flow corresponds to the \emph{generalized Ricci-flow} \cite{FernandezStreetsLibro,Oliynyk,Streets0}, a fact that we use to propose the Heterotic-Ricci flow as a higher-order curvature correction to the latter. 

%%%%%%%%%%%%%%%%%%%%%%%%%%%%%%%%%%%%%%%%%%%%%%%%%%%%%%%%%%%%%%%%%%%%%%%%%%%%

\subsection{Preliminaries}
\label{subsec:Preliminaries}

%%%%%%%%%%%%%%%%%%%%%%%%%%%%%%%%%%%%%%%%%%%%%%%%%%%%%%%%%%%%%%%%%%%%%%%%%%%%

Let $M$ be an oriented manifold $M$ of dimension $d$ equipped with a Riemannian metric $g$. We will denote by $\langle - , -\rangle_g$ the determinant inner product induced by $g$ on the exterior algebra bundle of $M$, and by $\vert - \vert_g$ its associated norm. For every linear connection $\nabla$ on $TM$, we will denote by $\mR^{\nabla}$ its curvature tensor, which in our conventions is defined by the following formula:
\begin{equation*}
\mR^{\nabla}_{u,v} w = \nabla_{u} \nabla_{v} w - \nabla_{v} \nabla_{u} w - \nabla_{[u,v]} w
\end{equation*}

\noindent
for every vector fields $u,v,w\in \mathfrak{X}(M)$. In the following, we will exclusively consider metric-compatible connections on the given Riemannian manifold $(M,g)$, and we will understand their curvature tensors $\mR^{\nabla}$ as sections of $\wedge^2 M\otimes \wedge^2 M$, upon identification of skew-symmetric endomorphisms with two-forms, using the Riemannian metric. Therefore, in our conventions the norm of $\mR^{\nabla}$ is explicitly given by:
\begin{eqnarray*}
\vert \mR^{\nabla} \vert_g^2 : = \langle \mR^{\nabla} , \mR^{\nabla} \rangle_g = \frac{1}{2} \sum_{i,j=1}^d \langle \mR^{\nabla}_{e_i,e_j} , \mR^{\nabla}_{e_i,e_j} \rangle_g = \frac{1}{4}\sum_{i,j,k,l=1}^d \mR^{\nabla}_{e_i,e_j}(e_k,e_l)\, \mR^{\nabla}_{e_i,e_j}(e_k,e_l)
\end{eqnarray*}

\noindent
in terms of any orthonormal frame $(e_i)$, where $\mR^{\nabla}_{e_i,e_j} = e_j \lrcorner e_i\lrcorner \mR^{\nabla}$ denotes evaluation of $e_i$ and $e_j$ in the \emph{first factor} $\wedge^2 M$ of $\mR^{\nabla}$ and $\mR^{\nabla}(e_i,e_j)$ denotes evaluation of $e_i$ and $e_j$ on the \emph{second factor} $\wedge^2 M$ of $\mR^{\nabla}$. For every three-form $H\in \Omega^3(M)$, we define the symmetric bilinear form:
\begin{eqnarray*}
(H\circ_g H)(v_1,v_2) :=   \langle v_1\lrcorner H , v_2\lrcorner H\rangle_g = \frac{1}{2} \sum_{i,j=1}^d H(v_1,e_i,e_j) H(v_2,e_i,e_j)\, , \quad v_1, v_2 \in \mathfrak{X}(M)\, .
\end{eqnarray*}

\noindent
Similarly, we define another symmetric bilinear form associated to the curvature tensor of $\nabla$ by:
\begin{equation*}
(\mR^{\nabla}\circ_g \mR^{\nabla})(v_1,v_2) := \langle v_1\lrcorner \mR^{\nabla}, v_2\lrcorner \mR^{\nabla}\rangle_g = \frac{1}{2}\sum_{i,j,k=1}^d \mR^{\nabla}_{v_1 e_i}(e_j,e_k) \mR^{\nabla}_{v_2 e_i}(e_j, e_k)\, , \quad v_1, v_2 \in \mathfrak{X}(M)\, .
\end{equation*} 

\noindent
Furthermore, associated to the curvature tensor of $\nabla$, we define the four-form:
\begin{equation*}
\langle \mR^{\nabla}\wedge \mR^{\nabla}\rangle_g = \frac{1}{2} \sum_{i,j=1}^d \mR^{\nabla}(e_i,e_j)\wedge \mR^{\nabla}(e_i,e_j)\, .
\end{equation*}

\noindent
 obtained by taking the wedge product on the \emph{first factor}  and the norm of the \emph{second factor} of $\mR^{\nabla}$ in $\wedge^2 M\otimes \wedge^2 M$. Given a three-form $H\in \Omega^3(M)$ and vector field $u\in \mathfrak{X}(M)$ on $(M,g)$, for ease of notation, we will occasionally denote by $H_u\in \Gamma(\End(TM))$ the skew-symmetric endomorphism of $TM$ defined by:
\begin{eqnarray*}
H_u(v) = H(u,v)^{\sharp_g}
\end{eqnarray*}

\noindent
where  $\sharp_g\colon T^{\ast}M \to TM$ is the musical isomorphism induced by $g$. We will denote by the same symbol the extension of $H_u$ to all tensor bundles over $M$ as derivation. In particular, if $\alpha$ is any $p$-form and $(e_i)$ some local orthonormal frame, 
$$H_u(\alpha)=\sum_{i=1}^d H_u(e_i)\wedge e_i\lrcorner \alpha.$$
Given $u, v \in \mathfrak{X}(M)$ the composition of $H_u$ and $H_v$ will be denoted by $H_u\circ H_v\in \Gamma(\End(TM))$, whereas their commutator as endomorphisms will be denoted by $[H_u, H_v] \in \Gamma(\End(TM))$. Similarly, given tensors $\tau_1 , \tau_2 \in \Gamma(T^{\ast}M\otimes T^{\ast}M)$, their commutator as endomorphisms obtained via the underlying Riemannian metric will be denoted simply by $[\tau_1, \tau_2] \in \Gamma(\End(TM))$. For every Riemannian metric $g$ and three-form $H$ on $M$ we define the connection $\nabla^{g,H}$ on the tangent bundle $TM$ as the unique metric  connection on $(M,g)$ with totally skew-symmetric torsion equal to $- H$. The connection $\nabla^{g,H}$ is explicitly given in terms of the Levi-Civita connection $\nabla^g$ of $g$ as follows:
\begin{equation*}
\nabla_{u}^{g,H} v = \nabla_{u}^g v - \frac{1}{2} H_u(v) = \nabla_{u}^g v - \frac{1}{2} H(u,v)^{\sharp_g}\, , \qquad \forall\,\, u
, v \in \mathfrak{X}(M).
\end{equation*}

\noindent
For ease of notation in the following the Riemannian curvature tensor of $g$ will be denoted by $\mR^g$, whereas the curvature tensor of a metric connection with skew-symmetric torsion $\nabla^{g,H}$ will be denoted by $\mR^{g,H}$. Recall the standard relation:
\begin{equation}\label{rgh}
\mR^{g,H}_{u\, v} = \mR^g_{u\, v} - \frac{1}{2} (\nabla^g_u H)(v) + \frac{1}{2} (\nabla^g_v H)(u) + \frac{1}{4} [H_u , H_v] 
\end{equation}

\noindent
between the curvature tensor $\mR^{g,H}$ of $\nabla^{g,H}$ and the Riemann tensor of $g$. Taking the trace of $\mR^{g,H}$ in the previous equation we obtain the Ricci tensor of $\nabla^{g,H}$, given explicitly by:
\begin{eqnarray*}
\mathrm{Ric}^{g,H} = \mathrm{Ric}^{g} - \frac{1}{2} H\circ_{g_t} H + \frac{1}{2} \delta^{g} H.
\end{eqnarray*}

\noindent
This equation will be useful in Section \ref{sec:HS} to understand the Heterotic soliton system.

%%%%%%%%%%%%%%%%%%%%%%%%%%%%%%%%%%%%%%%%%%%%%%%%%%%%%%%%%%%%%%%%%%%%%%%%%%%%

\subsection{The bosonic string action and the generalized Ricci flow}
\label{subsec:NSNSwordlsheet}

%%%%%%%%%%%%%%%%%%%%%%%%%%%%%%%%%%%%%%%%%%%%%%%%%%%%%%%%%%%%%%%%%%%%%%%%%%%%

Let $X$ be a compact and oriented real surface and let $M$ be an oriented manifold of dimension $d$. Given a Riemannian metric $g$ on $M$, a two-form $b\in \Omega^2(M)$, and a function $\phi\in  C^{\infty}(M)$, the \emph{bosonic string action} determined by the triple $(g,b,\phi)$ on the pair $(X,M)$ is the action functional:
\begin{equation*}
\cS \colon \Met(X)\times C^{\infty}(X,M) \to \mathbb{R}\, , 
\end{equation*}

\noindent
defined on the space of Riemannian metrics $\Met(X)$ on $X$ and the space of smooth maps $C^{\infty}(X,M)$ from $X$ to $M$ by the following action functional \cite{Polchinski:1998rq}:
\begin{equation*}
\cS[h,\Psi] = - \frac{1}{\kappa} \int_{X}  \left\{ \vert \dd \Psi\vert_{h,g}^2 + \ast_h (\Psi^{\ast} b ) - \kappa \, \phi\circ \Psi\, \mathrm{R}^h\right\}\,\nu_h \, ,
\end{equation*}

\noindent
where $h\in \Met(X)$, $\Psi\in C^{\infty}(X,M)$ and $\kappa$ is a positive real constant. Here $\mathrm{R}^h$ denotes the scalar curvature of $h$, $\dd\Psi\in \Omega^1 (X,TM^{\Psi})$ is the differential of $\Psi$ understood as a one-form on $X$ taking values in the pull-back $TM^{\Psi}$ of $TM$ by $\Psi$ and $\vert \dd \Psi\vert_{h,g}^2\in C^{\infty}(X)$ denotes the standard energy of the map $\Psi\colon X\to M$ computed with respect to the metrics $h$ and $g$. Therefore, the configuration space of the bosonic string action consists of pairs $(h,\Psi)$ given by Riemannian metrics $g$ on $X$ and maps $\Psi\colon X\to M$. This configuration space admits a large automorphism group of transformations that preserves the action functional $\cS$. In particular, $\cS$ is invariant under \emph{Weyl transformations}, namely conformal rescalings of $h$ by a positive real function. These symmetries are particularly important for the physical interpretation of the bosonic string action, as we will remark below.

The triple $(g,b,\phi)$ that determines the action functional $\cS$ is not \emph{dynamical} and is instead interpreted in this framework as the \emph{couplings} of the bosonic string action. Every such choice of data $(g,b,\phi)$ gives rise to a well-defined theory at the \emph{classical level}, that is, at the level of the classical equations of motion that are obtained via the first-order extremization of the action functional $\cS$. In contrast, not every triple $(g,b,\phi)$ leads to a bosonic string action admitting a consistent quantization. The quantization scheme of the bosonic string theory involves a regularization procedure that introduces an \emph{ultra-violet cutoff} $\lambda$ \cite{Polchinski:1998rq}. Through this procedure, physical quantities, in particular, the couplings of the theory, normally acquire a dependence on the scale $\lambda$, in which case the theory is no longer conformally invariant, that is, it is not invariant under Weyl transformations. This implies that Weyl transformations are not guaranteed to be symmetries of the bosonic string at the quantum level, something that cannot be allowed because of the physical consistency of the theory as a theory of \emph{quantum gravity}. The dependence of the couplings of $\cS$ on the renormalization scale $\lambda$ is controlled through the \emph{renormalization group flow equations}, which in the present case are given by:
\begin{equation*}
\frac{\partial g_t}{\partial t} = - \beta_{g_t}\, , \qquad \frac{\partial b_t}{\partial t} = - \beta_{b_t}\, , \qquad \frac{\partial \phi_t}{\partial t} = - \beta_{\phi_t}\, ,
\end{equation*}

\noindent
where $t\in \mathbb{R}$ is the logarithm of the renormalization scale, $(g_t,b_t,\phi_t)$ denotes a one-parameter family of Riemannian metrics, two-forms and functions on $M$ and:
\begin{equation*}
\beta_{g_t}\in \Gamma(T^{\ast}M\odot T^{\ast}M)\, , \quad \beta_{b_t}\in \Omega^2(M)\, ,\quad \beta_{\phi_t}\in C^{\infty}(M)\, ,
\end{equation*}

\noindent
denote the \emph{beta functionals} of $g_t$, $b_t$, and $\phi_t$, respectively, where the latter are being understood as families of coupling fields of the bosonic string action. For Weyl invariance to be preserved in the quantum theory we must have $\beta_g = \beta_b = \beta_{\phi} = 0$ possibly modulo a time-dependent diffeomorphism. Computing the beta functionals of the bosonic string is a complicated task that is usually performed perturbatively in the constant $\kappa$. Let $\mathrm{Ric}^{g_t}$ denote the Ricci tensor of $g_t$. To the lowest order in $\kappa$ we have \cite{Oliynyk,Polchinski:1998rq}:
\begin{eqnarray*}
& \frac{\partial g_t}{\partial t} = -  \kappa (\mathrm{Ric}^{g_t}   - \frac{1}{2} H_t\circ H_t) + o(\kappa^2)\, , \\
& \frac{\partial b_t}{\partial t} = -  \frac{\kappa}{2} \,  \delta^{g_t} H_t + o(\kappa^2)\, , \quad \frac{\partial \phi_t}{\partial t} = c +  \frac{\kappa}{4}\, (-\delta^{g_t}\varphi_t + \vert H_t\vert_{g_t}^2 ) + o(\kappa^2)\, ,
\end{eqnarray*}

\noindent
where $c$ is a constant that depends on the dimension of $M$ and we have set $H_t := \dd b_t$ and $\varphi_t := \dd \phi_t$. Assuming $c=0$, which we will do freely in the following, the previous equations define a weakly parabolic flow whose self-similar solutions are solutions of the bosonic sector of NS-NS supergravity on $M$. By virtue of the evolution equation satisfied by $b_t$, we obtain the following evolution equation for $H_t$:
\begin{equation*}
\frac{\partial H_t}{\partial t} = -  \frac{\kappa}{2} \, \dd\delta^{g_t} H_t + o(\kappa^2)\, .
\end{equation*} 

\noindent
This equation is sometimes considered instead of the evolution equation for $b_t$ as part of the renormalization group flow equations of the NS-NS worldsheet at first order in $\kappa$. Note that the evolution equation for $\phi_t$ decouples and therefore can be considered separately. The \emph{generalized Ricci flow} \cite{FernandezStreetsLibro} can be introduced as the first-order renormalization group flow for $(g_t,H_t)$, which after an appropriate \emph{time} rescaling by $\kappa/2$ is given by the following system of evolution equations:
\begin{equation*}
\frac{\partial g_t}{\partial t} = - 2  \mathrm{Ric}^{g_t}   + \frac{1}{2} H_t\circ H_t\, , \quad \frac{\partial H_t}{\partial t} = -  \dd\delta^{g_t} H_t \, ,
\end{equation*}

\noindent
for pairs $(g_t,H_t)$.

%%%%%%%%%%%%%%%%%%%%%%%%%%%%%%%%%%%%%%%%%%%%%%%%%%%%%%%%%%%%%%%%%%%%%%%%%%%%

\subsection{The Heterotic extension of the generalized Ricci flow}
\label{subsec:HigherOrder}

%%%%%%%%%%%%%%%%%%%%%%%%%%%%%%%%%%%%%%%%%%%%%%%%%%%%%%%%%%%%%%%%%%%%%%%%%%%%

The first-order renormalization group flow equations of the bosonic string receive higher-order corrections in $\kappa$ \cite{BeckerBeckerSchwarz}. To compute these corrections, the bosonic string needs to be considered as a subsector of a particular string theory, which can be the bosonic string itself or any of the five superstring theories since all of them have the bosonic string as a common subsector. Considering the bosonic string as the NS-NS truncation of Heterotic string theory, a careful inspection of the effective action of Heterotic supergravity expanded perturbatively in the string slope parameter $\kappa$ together with the computation of the beta functionals of the Heterotic (1,0) worldsheet leads to the following renormalization group flow equations for $(g_t,H_t,\phi_t)$ at second-order in $\kappa$ \cite{BRI,BRII,Metsaev:1987bc,Metsaev:1987zx}:
\begin{eqnarray*}
& \frac{\partial g_t}{\partial t} = -\kappa\, (\mathrm{Ric}^{g_t} - \frac{1}{2} H_t\,\,\circ_{g_t} H_t) - \kappa^2\, \mR^{g_t,H_t} \circ_{g_t}\mR^{g_t,H_t} + o(\kappa^3)\, , \\
& \frac{\partial H_t}{\partial t} = -  \frac{\kappa}{2} \, \dd\delta^{g_t} H_t + o(\kappa^3)\, , \quad  \frac{\partial \phi_t}{\partial t} = \frac{\kappa}{4}\, (\vert H_t\vert_{g_t}^2 -\delta^{g_t}\dd\phi_t -    \kappa\, |\mR^{g_t,H_t}|^2_{g_t}) + o(\kappa^3)\, .
\end{eqnarray*}

\noindent
together with the Bianchi identity for $H_t$:
\begin{equation}
\label{eq:Bianchi}
\dd H_t + \kappa\, \langle\mR^{g_t,H_t} \wedge \mR^{g_t,H_t}\rangle_{g_t}=0\, .
\end{equation}

\noindent
As in the case of the renormalization group flow equations at first order, the evolution equation for the dilaton $\phi_t$ decouples from the evolution equations of $(g_t,H_t)$ and therefore can be considered separately. Proceeding analogously to the generalized Ricci-flow case, we define the \emph{Heterotic-Ricci flow equations} as the renormalization group flow equations of the NS-NS sector of the Heterotic string at second order in $\kappa$, that is:

\begin{definition}
\label{def:Heterotic}
The \emph{Heterotic-Ricci flow} on $M$ is the following differential system:
\begin{eqnarray*}
& \partial_t g_t = -2 \, (\mathrm{Ric}^{g_t}  - \frac{1}{2} H_t\,\,\circ_{g_t} H_t) -  2 \kappa \,\mR^{g_t,H_t} \circ_{g_t} \mR^{g_t,H_t} \, , \quad  \partial_t H_t = - \dd\delta^{g_t} H_t\, , \\ 
& \dd H_t +\kappa\, \langle \mR^{g_t,H_t} \wedge \mR^{g_t,H_t}\rangle_{g_t} = 0 
\end{eqnarray*}

\noindent
for families $(g_t,H_t)$ of Riemannian metrics and three forms on $M$.  
\end{definition}

\noindent
We will refer to solutions to the Heterotic-Ricci flow simply as \emph{Heterotic-Ricci flows}. Clearly, when $\kappa = 0$ we recover the generalized Ricci flow whereas when $H_t=0$ the Heterotic-Ricci flow reduces to the RG-2 flow \cite{GGI} subject to the intriguing constraint:
\begin{equation*}
\kappa\, \langle\mR^{g_t,H_t} \wedge \mR^{g_t,H_t}\rangle_{g_t}=0 
\end{equation*}

\noindent
which to the best of our knowledge has not been studied in the RG-2 flow literature. It would be interesting to investigate precisely how the previous equation constrains the RG 2-flow and how it affects its long-time behavior and existence. We will refer to solutions $(g_t,H_t)$ to the Heterotic-Ricci flow equations simply as Heterotic-Ricci flows. Note that a Heterotic-Ricci flow $(g_t,H_t)$ can be completed into a solution of the two-loop renormalization group flow equations of the NS-NS sector of the Heterotic string if and only if there exists a closed one-form $\varphi_0\in \Omega^1(M)$ and a family of functions $\phi_t$ such that the following equation, to which we will refer as the \emph{dilaton flow equation}, is satisfied:
\begin{equation}
\label{eq:flowdilaton}
\frac{\partial \phi_t}{\partial t} = \frac{1}{2}\, (\vert H_t\vert_{g_t}^2 -\delta^{g_t}\varphi_t -    \kappa\, |\mR^{g_t,H_t}|^2_{g_t})\, ,
\end{equation}

\noindent
where we have set $\varphi_t = \varphi_0 + \dd \phi_t$. Note that if $\kappa =0$ the previous equation reduces to the dilaton flow as defined in \cite{StreetsII}. It would be very interesting to see if the dilaton flow with $\kappa\neq 0$ can have the same type of applications for the Heterotic-Ricci flow as it does when $\kappa = 0$ for the generalized Ricci flow, as explained in \cite{StreetsII}.

\begin{definition}
Let $(g_t,H_t)$ be a Heterotic-Ricci flow, and let $\sigma \in H^1(M,\mathbb{R})$ be a cohomology class. A dilaton for $(g_t,H_t)$ with Lee class $\sigma$ is a closed one-form $(\varphi_0)$ on $M$ such that $\sigma = [\varphi_0]$ and a family of functions $(\phi_t)$ which solves the dilaton flow equation \eqref{eq:flowdilaton} for $\varphi_t := \varphi_0 + \dd\phi_t$.  
\end{definition}

\noindent
Note that the term $\mR^{g_t,H_t} \circ_{g_t} \mR^{g_t,H_t}$ involves quadratic terms in the Riemann tensor of $g_t$ as well as terms involving covariant derivatives of $H_t$ with respect to the Levi-Civita connection $\nabla^{g_t}$. On the other hand, given a Heterotic-Ricci flow $(g_t,H_t)$ the dilaton evolution equation becomes an inhomogeneous linear heat equation for $\phi_t$ on $M$, and therefore standard theory applies regarding the existence and uniqueness of its solutions.

\begin{proposition}
Let $(g_t,H_t)$ be a Heterotic-Ricci flow on a compact manifold $M$ and defined on the interval $\cI$. Then, for each cohomology class $\sigma\in H^1(M,\mathbb{R})$ there exists a unique dilaton $\varphi_t$ with Lee class $\sigma$ associated to $(g_t,H_t)$. 
\end{proposition}

\begin{proof}
Let $(g_t,H_t)$ be a Heterotic-Ricci flow. Given a cohomology class $\sigma \in H^1(M,\mathbb{R})$, fix a closed one-form $\varphi_0\in \Omega^1(M)$ such that $\sigma = [\varphi_0]$ and write $\varphi_t = \varphi_0 + \dd\phi_t$ in terms of a family of functions $\phi_t$. Plugging $\varphi_t = \varphi_0 + \dd\phi_t$ in the dilaton equation \eqref{eq:flowdilaton} we obtain:
\begin{equation*}
\frac{\partial \phi_t}{\partial t} = \frac{1}{2}\, ( - \Delta_{g_t}\phi_t + \vert H_t\vert_{g_t}^2  -    \kappa\, |\mR^{g_t,H_t}|^2_{g_t} - \delta^{g_t}\varphi_0 )\, ,
\end{equation*}

\noindent
where $(g_t,H_t)$ is considered as given data and $\Delta_{g_t} = (\dd\delta^{g_t} + \delta^{g_t}\dd)$ denotes the Laplace-de-Rham operator. This is a linear parabolic equation for $(\phi_t)$ and by \cite[Theorem 4.47]{Aubin} it admits a unique smooth solution on $\cI$ such that $\phi_0 = 0$.
\end{proof}

\begin{remark}
If $\varphi_t = \varphi_0 + \dd\phi_t$ and $\varphi_t^{\prime} = \varphi^{\prime}_0 + \dd\phi^{\prime}_t$ are two dilatons with the same Lee class and associated to the same Heterotic-Ricci flow on a compact manifold $M$, then $\phi_t^{\prime} = \phi_t + f$ for a function $f$ on $M$ unique modulo an additive constant.
\end{remark}

\noindent
By the previous proposition, we will not be concerned with the dilaton equation in the sequel. 

\begin{example}
\label{ep:2dHR}
When $M$ is two-dimensional we have $H_t=0$ identically and the Heterotic-Ricci flow equations reduce to the RG-2 flow equations \cite{GGI}:
\begin{eqnarray*}
& \frac{\partial g_t}{\partial t} = -2 \,  \mathrm{Ric}^{g_t}  -  2 \kappa \, \mathrm{R}^{g_t} \circ_{g_t}\mathrm{R}^{g_t} \, ,  
\end{eqnarray*}

\noindent
where $\mathrm{R}^{g_t}$ is the Riemann tensor of $g_t$. As explained in \cite{Oliynyk:2009rh}, in this dimension the RG-2 flow equations simplify to:
\begin{equation*}
\frac{\partial g_t}{\partial t} = - (s_{g_t} + \frac{\kappa}{2}s_{g_t}^2) \, g_t\, ,
\end{equation*}

\noindent
where $s_{g_t}$ denotes the scalar curvature of $g_t$. We refer the reader to \cite{Oliynyk:2009rh} for more details. 
\end{example}

\begin{remark}
\label{remark:stringstructure}
The proper \emph{gauge theoretic} formulation of the Heterotic-Ricci flow in terms of the \emph{b-field} instead of the three-form flux $H$ requires understanding the Heterotic-Ricci flow as a system of differential equations either on a string algebroid, as a particular case of \cite{Garcia-Fernandez:2016ofz}, or on a string structure on the frame bundle of $M$, see \cite{Killingback,Redden,Waldorf} and references therein for more details. This perspective is crucial to properly understand the groupoid of automorphisms of the flow equations as well as its moduli of solitons. 
\end{remark} 
 
%%%%%%%%%%%%%%%%%%%%%%%%%%%%%%%%%%%%%%%%%%%%%%%%%%%%%%%%%%%%%%%%%%%%%%%%%%%%

\subsection{Heterotic-Ricci flows on three-manifolds}

%%%%%%%%%%%%%%%%%%%%%%%%%%%%%%%%%%%%%%%%%%%%%%%%%%%%%%%%%%%%%%%%%%%%%%%%%%%%

As remarked in Example \ref{ep:2dHR}, in two dimensions the Heterotic-Ricci flow reduces to the RG-2 flow \cite{GGI}, a curvature flow that is well-known to arise through a second-order correction to the Ricci-flow as prescribed by the renormalization group flow of a metric non-linear sigma model. In three dimensions the Heterotic-Ricci flow equations reduce to:
\begin{eqnarray}
\label{eq:3dHRflow}
\partial_t g_t  = -2 \, (\mathrm{Ric}^{g_t}  - \frac{1}{4} H_t\circ H_t) -  2 \kappa \, \mR^{g_t,H_t} \circ_{g_t}\mR^{g_t,H_t} \, , \quad  \partial_t H_t  = - \dd\delta^{g_t} H_t\, ,  
\end{eqnarray}

\noindent
for families $(g_t,H_t)$ of Riemannian metrics and three-forms $H_t$ on $M$. In particular, $H_t$ is closed and the Bianchi identity \eqref{eq:Bianchi} is automatically satisfied for dimensional reasons. Suppose that $H_t$ evolves within the de Rham cohomology class that it determines at $t_0$. We can write:
\begin{equation*}
H_t = H_0 + \dd b_t\, .
\end{equation*}

\noindent
With this assumption equation \eqref{eq:3dHRflow} is equivalent to:
\begin{equation*}
\frac{\partial b_t}{\partial t} = -\delta^{g_t} H_t + \omega_t
\end{equation*}

\noindent
for a family of closed two-forms $(\omega_t)$. Setting this spurious family $(\omega_t)$ of two forms to zero  we obtain the following refined formulation of the Heterotic-Ricci flow:
\begin{eqnarray}
\label{eq:3dHRflowbfield}
\partial_t (g_t + b_t) = -2 \mathrm{Ric}^{g_t,H_t}   -  2 \kappa \, \mR^{g_t,H_t} \circ_{g_t}\mR^{g_t,H_t}  
\end{eqnarray}

\noindent
which now involves a family of Riemannian metrics $(g_t)$ and two-forms $(b_t)$. Indeed, every family $(g_t,b_t)$ solving the previous flow equations naturally yields a solution of equations \eqref{eq:3dHRflow}, although the converse may not be true. This is related to the fact that the proper, gauge-theoretic, formulation of the Heterotic-Ricci flow involves the notion of \emph{b-field} and occurs in terms of connections on a string structure, see Remark \ref{remark:stringstructure}. This point of view requires the development of novel Riemannian tools on string structures and will be considered elsewhere. Here instead we will simplify equations \eqref{eq:3dHRflow} using the fact that the dual of a three-form is a function and the fact that the Riemann tensor is completely determined by the Ricci curvature. Given a pair $(g_t , H_t)$, we write $H_t= f_t \nu_{g_t}$ in terms of the unique family of functions $f_t = \ast_{g_t} H_t\in C^{\infty}(M)$, where $\nu_{g_t}$ denotes the Riemannian volume form associated $g_t$. Consequently, when $M$ is three-dimensional we define the configuration space $\Conf(M)$ of the Heterotic-Ricci flow system on $M$ as the set of tuples $(g_t,f_t)$, where $g_t$ is a family of Riemannian metrics and $f_t$ is a family of functions. Similarly, we define the set of Heterotic-Ricci flows $\Sol_{\kappa}(M)$ on $M$ as the set of pairs $(g_t,f_t)\in\Conf(M)$ such that $(g_t,H_t = \ast_{g_t} f_t)$ is a Heterotic-Ricci flow with parameter $\kappa\in \mathbb{R}$. Using Lemma \ref{lemma:vg} in Appendix \ref{app:curvature3d} we can write the three-dimensional Heterotic-Ricci flow as an evolution differential system for a family of metrics $g_t$ and a family of functions $f_t$ involving exclusively the Ricci tensor and scalar curvature of $g_t$ as the only curvature operators occurring in the system.

\begin{proposition}
\label{prop:eqs3hrf}
A pair $(g_t, f_t)\in \Conf(M)$ is a Heterotic-Ricci flow if and only if it satisfies the following evolution equations:
\begin{eqnarray}
\label{eq:3dHRflowreformulated1}
& \partial_t g_t = 2\kappa\,  \mathrm{Ric}^{g_t} \circ \mathrm{Ric}^{g_t} - (2 + \kappa (2 s^{g_t} - f^2_t) ) \mathrm{Ric}^{g_t}\nonumber \\ 
& +(f_t^2 + \kappa ((s^{g_t})^2 -2 \vert \mathrm{Ric}^{g_t} \vert_{g_t}^2  - \frac{1}{2}\vert \dd f_t \vert_{g_t}^2 - \frac{1}{4} f^4_t) )g_t - \kappa  [\ast_{g_t}\dd f_t,\mathrm{Ric}^{g_t}]  - \frac12\kappa \,\dd f_t \otimes \dd f_t \\
& \partial_t f_t + \frac{1}{2} \mathrm{Tr}_{g_t}(\partial_t g_t) f_t  +  \Delta_{g_t} f_t = 0\, ,\label{eq:3dHRflowreformulated2}
\end{eqnarray}

\noindent
where $f_t = \ast_{g_t} H_t$ and $\Delta_{g_t} = \delta^{g_t}\dd \colon C^{\infty}(M) \to C^{\infty}(M)$ is the Laplace operator on functions.
\end{proposition}

\begin{proof}
By definition $(g_t, f_t)\in \Conf(M)$ is a Heterotic-Ricci flow if and only if $(g_t,H_t = f_t \nu_{g_t})$ satisfies equations \eqref{eq:3dHRflow}. Plugging Equation \eqref{eq:lemmavg} into the first equation in \eqref{eq:3dHRflow} and using the identity:
\begin{equation*}
H_t\circ_{g_t} H_t(v_1,v_2) = g_t(v_1,v_2) f^2_t\, ,
\end{equation*}
	
\noindent
it follows that the first equation in \eqref{eq:3dHRflow} is equivalent to Equation \eqref{eq:3dHRflowreformulated1}. For the second equation in \eqref{eq:3dHRflow} we compute:
\begin{equation*}
\partial_t H_t  = \partial_t (f_t \nu_{g_t}) = (\partial_t f_t) \nu_{g_t} + \frac{1}{2} \mathrm{Tr}_{g_t}(\partial_t g_t) f_t \nu_{g_t}\, ,
\end{equation*}
	
\noindent
whence the second equation in \eqref{eq:3dHRflow} is equivalent to:
\begin{eqnarray*}
(\partial_t f_t) \nu_{g_t} + \frac{1}{2} \mathrm{Tr}_{g_t}(\partial_t g_t) f_t \nu_{g_t} =  \dd \ast_{g_t} \dd f_t\, .
\end{eqnarray*}

\noindent
Taking the Hodge dual of this equation we obtain \eqref{eq:3dHRflowreformulated2} and we conclude.
\end{proof} 

%%%%%%%%%%%%%%%%%%%%%%%%%%%%%%%%%%%%%%%%%%%%%%%%%%%%%%%%%%%%%%%%%%%%%%%%%%%%
 
\subsection{Homothety flows}
\label{subsec:exampleHeteroticRicci}
 
%%%%%%%%%%%%%%%%%%%%%%%%%%%%%%%%%%%%%%%%%%%%%%%%%%%%%%%%%%%%%%%%%%%%%%%%%%%%

 In this section, we consider a particular class of three-dimensional Heterotic Ricci flows for which the metric evolves by homotheties. Note that,  in contrast with the standard Ricci flow, the Heterotic-Ricci flow is not scale-invariant unless the constant $\kappa$ is rescaled too. Doing this for time-dependent rescalings would introduce a time dependence in $\kappa$. This is a priori not allowed but can be a convenient mechanism to consider, see \cite{CarforaGuenther} for a detailed discussion of this issue in the context of the RG 2-flow. As a consequence, Heterotic-Ricci flows that evolve by homotheties can be non-trivial and may become remarkably complicated. Suppose then that $(g_t ,f_t)$ satisfies:
 \begin{equation*}
g_t =\sigma_t g \, , \qquad \dd f_t = 0\, ,
 \end{equation*}

 \noindent
 where $(\sigma_t)$ is a family of constants and $g$ is a Riemannian metric on $M$. We will call such pairs $(\sigma_t g, f_t)$ \emph{homothety flows}. A direct computation by substitution yields the following characterization of the homothety flows that solve the Heterotic-Ricci flow equations.

 \begin{lemma}
A homothety flow $(\sigma_t g , f_t)$ is a Heterotic-Ricci flow if and only if:
\begin{equation}
\label{eq:homotecyeq}
\frac{1}{2}\partial_t \sigma^2_t\, g = 2\kappa \,\sigma_t\,  \mathrm{Ric}^{g} \circ \mathrm{Ric}^{g} - \left (2 \sigma_t + \kappa \left ( 2 s_{g}  - \frac{\mu^2}{\sigma^2_t}\right ) \right) \mathrm{Ric}^{g} + \left( \kappa \left ( s_g^2  -2  \vert \mathrm{Ric}^{g} \vert_{g}^2  -  \frac{\mu^4}{4 \sigma^4_t}\right ) + \frac{\mu^2}{\sigma_t} \right) \, g  
\end{equation}
\noindent
and $f_t = \mu \sigma_t^{-\frac{3}{2}}$ for a real constant $\mu\in\mathbb{R}$.
 \end{lemma}
 
 \noindent
The \emph{initial point} of a homothety flow cannot be arbitrary, as we are showing next. This is the Heterotic-Ricci flow analog of \cite[Theorem 3]{GGI} that was proven in Op. Cit. for the RG 2-flow. 

\begin{proposition}
\label{prop:eins}
Let $(g_t=\sigma_t g , f_t = \mu \sigma_t^{-\frac{3}{2}})$  be a homothety flow solving the Heterotic-Ricci flow with non-constant $\sigma_t$. If $s_g^2+\mu^2\neq 0$, then $(M,g)$ is Einstein.
\end{proposition}

\begin{proof}
We can diagonalize $\mathrm{Ric}^g$ in Equation \eqref{eq:homotecyeq}, obtaining:
\begin{equation*}
\sigma_t\partial_t \sigma_t  = 2\kappa \,\sigma_t\,  \lambda^2_i - \left (2 \sigma_t + \kappa \left ( 2 s_{g}  - \frac{\mu^2}{\sigma^2_t}\right ) \right ) \lambda_i + \left( \kappa \left( s_g^2  -2  \vert \mathrm{Ric}^{g} \vert_{g}^2  -  \frac{\mu^4}{4 \sigma^4_t}\right) + \frac{\mu^2}{\sigma_t} \right)    \, ,
\end{equation*}

\noindent
where $\lambda_i$, $i=1,2,3$,  are principal Ricci curvatures of $g$. Subtracting  the previous equations by pairs, we obtain:  
\begin{equation*}
  (\lambda^2_1 - \lambda^2_2) = F_t (\lambda_1 - \lambda_2) \, , \quad (\lambda^2_1 - \lambda^2_3) = F_t (\lambda_1 - \lambda_3)   \, ,\quad (\lambda^2_2 - \lambda^2_3) = F_t (\lambda_2 - \lambda_3)     \, ,
\end{equation*}
\noindent
where:
\begin{equation*}
F_t := \frac{1}{\kappa}  +   \left ( \frac{s_{g}}{\sigma_t}  - \frac{\mu^2}{2\sigma^3_t}\right )  \, .
\end{equation*}
\noindent
If either $s_g\neq 0$ or $\mu \neq 0$ then $F_t$ is not constant and we obtain $\lambda_1 = \lambda_2 = \lambda_3$. %The case $s_g=0$ and $\mu =0$ with $g$ non-flat can be shown to be inconsistent by a direct inspection of Equation \eqref{eq:homotecyeq}.
\end{proof}
\begin{remark}
\label{rem:ne}
Condition $s_g^2+\mu^2 \neq 0$ in Proposition \ref{prop:eins} is necessary. In fact, consider the Riemannian Lie group $(\mathrm{SU}(2),g)$ for which the following the following left-invariant frame $\{e_1,e_2,e_3\}$ is orthonormal:
\begin{equation*}
    [e_2,e_3]=\frac{1}{2 \sqrt{\kappa}} e_1\,, \quad [e_3,e_1]=\frac{1}{2 \sqrt{\kappa}} e_2\,, \quad [e_1,e_2]=\frac{2}{\sqrt{\kappa}}e_3\,.
\end{equation*}
Direct computation reveals that $\mathrm{Ric}^g=-\kappa^{-1} g+ 3 \kappa^{-1} e_3 \otimes e_3$. Clearly, the eigenvalues of $\mathrm{Ric}^g$ are $\lambda_1=\lambda_2=-\kappa^{-1}$ and $\lambda_3=2\kappa^{-1}$, so that $s_g=0$. If $W$ is the Lambert W function, defined as the inverse of the function $p(x)=x e^x$ when restricted to the interval $x \in (-e^{-1},+\infty)$, it can be checked that:
\begin{equation*}
\sigma_t=3+3 W\left (-\frac{2}3 \mathrm{Exp}\left[ \frac 23 \left (2t/\kappa-1\right ) \right ] \right)
\end{equation*}
 solves Equation \eqref{eq:homotecyeq} for $s_g=\mu=0$ and for the choice $(\mathrm{SU}(2),g)$, which is not Einstein. Observe that this particular $\sigma_t$ is defined for $t \in (-\infty,t_{\mathrm{max}})$, where $t_{\mathrm{max}}=\kappa/4(\log(27/8)-1)$.
\end{remark}
\noindent
Assuming that $g$ is Einstein in Equation \eqref{eq:homotecyeq} we obtain the following ordinary differential equation:
 \begin{equation}
 \label{eq:sigmatgeneral}
 \frac{1}{2}\partial_t \sigma^2_t = \left (\frac{2\kappa\, s_g}{3} - 2\right ) \frac{s_g}{3} \sigma_t - \frac{\kappa\, s^2_g}{3} + \frac{\mu^2}{\sigma_t} - \frac{\kappa\, s_g \,\mu^2}{\sigma_t^2} - \frac{\kappa \,\mu^4}{4 \sigma^4_t}    
 \end{equation}

 \noindent
 controlling the evolution of $\sigma_t$. If we set $\kappa = 0$ and $\mu=0$ we recover the standard linear behavior of $\sigma_t$ that corresponds to the Ricci flow case. If we just set $\mu=0$ we obtain the differential equation that controls $\sigma_t$ according to the RG-2 flow, namely:
 \begin{equation*}
\frac{1}{2}\partial_t \sigma^2_t = \left (\frac{2\kappa\, s_g}{3} - 2\right ) \frac{s_g}{3} \sigma_t - \frac{\kappa\, s^2_g}{3} \, .      
 \end{equation*}

 \noindent
Although the standard existence and uniqueness theorems for ordinary differential equations guarantee the existence of a solution to \eqref{eq:sigmatgeneral} in some short \emph{time} interval, the equation is in the general case too complex to admit a solution by elementary functions.  After possibly rescaling the metric $g$, we can set $s_g=1$, $s_g = 0$ or $s_g =-1$ obtaining respectively, the \emph{positive} homothety flow:
\begin{equation}
 \label{eq:sigmatpositive}
 \frac{1}{2}\partial_t \sigma^2_t = 2\,\left (\frac{\kappa}{9} - \frac{1}{3}\right ) \sigma_t - \frac{\kappa}{3} + \frac{\mu^2}{\sigma_t} - \frac{\kappa\,\mu^2}{\sigma_t^2} - \frac{\kappa \,\mu^4}{4 \sigma^4_t}    \, .
 \end{equation}
 \noindent
 the \emph{flat} homothety flow:
 \begin{equation}
 \label{eq:sigmatflat}
 \frac{1}{2}\partial_t \sigma^2_t  =   \frac{\mu^2}{\sigma_t} -  \frac{\kappa \,\mu^4}{4\,\sigma^4_t}  \, ,
 \end{equation}
 and the \emph{negative} homothety flow:
 \begin{equation}
 \label{eq:sigmatnegative}
 \frac{1}{2}\partial_t \sigma^2_t = 2\,\left (\frac{\kappa}{9} + \frac{1}{3}\right)\sigma_t - \frac{\kappa}{3} + \frac{\mu^2}{\sigma_t} + \frac{\kappa \,\mu^2}{\sigma_t^2} - \frac{\kappa \,\mu^4}{4 \sigma^4_t}    \, .
 \end{equation}
 \noindent
 It is interesting to note that whereas in the flat case $s_g = 0$ both the Ricci flow and RG-2 flow produce trivial homothety flows, the homothety Heterotic-Ricci flow is non-trivial and is instead prescribed by Equation \eqref{eq:sigmatflat}. This equation perfectly illustrates a phenomenon that is a common theme of the Heterotic-Ricci flow, that is, a \emph{positive} contribution corresponding to the term appearing in the generalized Ricci flow together with an opposite \emph{negative} contribution corresponding to the higher-order correction introduced by the Heterotic-Ricci flow.  

\begin{remark}
\label{remark:positivenegative}
Although the terms \emph{positive} and \emph{negative} in the previous paragraph have no specific meaning, in the original Lorentzian formulation of the system theory, they can be traced back to the fact that the higher-order correction $\mR^{g,H}\circ_g\mR^{g,H}$ present in the Heterotic system gives a \emph{negative energy} contribution that is of opposite sign to the contribution given by the other terms of the system. 
\end{remark}

\noindent
Let us study separately the positive, flat and negative cases.

\subsubsection{Positive homothety flows}
Finding explicit solutions to Equation \eqref{eq:sigmatpositive} is quite challenging, but it is possible to examine some qualitative and numerical features of its solutions. The crucial object to study is the right-hand side of Equation \eqref{eq:sigmatpositive}, which determines the derivative of $\sigma_t$:
\begin{equation*}
    F_p(\kappa,\mu,y)=2\,\left (\frac{\kappa}{9} - \frac{1}{3}\right ) y - \frac{\kappa}{3} + \frac{\mu^2}{y} - \frac{\kappa\,\mu^2}{y^2} - \frac{\kappa \,\mu^4}{4 y^4}    \, .
\end{equation*}
The first aspect to analyze is the existence of static solutions, i.e., $\sigma_t=1$ for every $t \in \mathbb{R}$. Setting $F_p(\kappa,\mu,1)=0$, we find a one-parameter family of values of $(\kappa,\mu)$ satisfying such condition, which is given by:
\begin{equation*}
    \kappa(\mu)=\kappa_{\mathrm{crit}}^p(\mu)=\frac{36 \mu^2-24}{9\mu^2(\mu^2+4)+4}\,.
\end{equation*}
Note that $\kappa_{\mathrm{crit}}^p(\mu) \geq 0$ if and only if $\mu^2 \geq 2/3$. Whenever $\kappa \neq \kappa_{\mathrm{crit}}^p(\mu)$, the solution is no longer static. We observe that:
\begin{itemize}
    \item When considering pairs $(\kappa,\mu)$ such that $0<\kappa<\kappa_{\mathrm{crit}}^p(\mu)$, it turns out that $F_p(\kappa,\mu,1)>0$. Sticking to these values of $(\kappa,\mu)$, one notices that there exist $y_1 \in (0,1)$ and $y_2>1$ such that  $F_p(\kappa,\mu,y_1)=F_p(\kappa,\mu,y_2)=0$. Since the vanishing of $F_p(\kappa,\mu,t)$ implies the vanishing of $\partial_t\sigma_t$, we conclude that all solutions $\sigma_t$ to Equation \eqref{eq:sigmatpositive} with $\kappa < \kappa_{\mathrm{crit}}^p(\mu)$ are well-defined for all times $t \in \mathbb{R}$, satisfying that $1> \lim_{t \rightarrow -\infty} \sigma_t>0$ and $+\infty > \lim_{t \rightarrow +\infty} \sigma_t>1$. Therefore, they define eternal and regular (in the sense that they asymptote to finite values as $t \rightarrow \pm \infty$) flows, which do not exist for the Ricci flow case $\kappa=\mu=0$. Some solutions of this type are illustrated in Figure \ref{fig:graphs} (left). 
     \item For pairs $(\kappa,\mu)$ such that $\kappa> \mathrm{max}(0,\kappa_{\mathrm{crit}}^p(\mu))$, a more detailed examination is required:
     \begin{itemize} %\footnote{This equation is obtained by setting $F_p(\kappa,\mu,y)=\partial_y F_p(\kappa,\mu,y)=0$.}
         \item If $\mu^2$ is greater than the unique real and positive solution of the algebraic equation\footnote{This equation is obtained after setting $F_p=\partial_y F_p=0$, imposing $y=1$ and identifying $x=\mu^2$.} $27 x^3 +6 x^2-68 x-8=0$, then all solutions $\sigma_t$ are seen to collapse at finite time. This is due to the fact that $F_p(\kappa,\mu,1)<0$, while $F_p(\kappa,\mu,y)$ decreases as $y$ goes to zero (so that $\sigma_t$ approaches zero at some finite time).
         \item Otherwise, there is an interval $\kappa \in (\mathrm{max}(0,\kappa_{\mathrm{crit}}^p(\mu)),\kappa_0(\mu))$ for which the solution is seen to exist for all times (diverging nonetheless  as $t \rightarrow -\infty$, while remaining finite for $t\rightarrow +\infty$), thus yielding an eternal flow. The value $\kappa_0(\mu)$ can be found by solving the equations:
         \begin{equation*}
             F_p(\kappa_0(\mu),\mu,y_0)=0\, , \quad \partial_{y} F_p(\kappa_0(\mu),\mu,y_0)=0\,,
         \end{equation*}
         for $y_0 \in (0,1)$. If $\kappa>\kappa_0(\mu)$, then solutions collapse at finite time.
     \end{itemize}
     On the other hand, if $\mu=0$ and $\kappa>0$, the solution is seen to collapse at finite time, while for $\mu=\kappa=0$ we have that $\sigma_t=-2t/3+1$. Finally, for $\kappa=0$, if $\mu^2<2/3$ we have that the flow is eternal, diverging as $t \rightarrow -\infty$ and remaining finite as $t \rightarrow +\infty$, while it collapses at finite time if $\mu^2>2/3$ (the case $\mu^2=2/3$ corresponds to the static one).
     %On the other hand, when $\kappa=0$ all solutions $\sigma_t$ turn out to collapse at finite time.
\end{itemize}

\subsubsection{Flat homothety flows}
In this case, it is possible to find an explicit solution to Equation \eqref{eq:sigmatflat}. Assuming $\mu\neq 0$ and defining $b=4/(\kappa \mu^2)-1$, the solution can be seen to be:
\begin{equation}
    \sigma_t=\left\lbrace \begin{matrix}\left ( \frac{\kappa \mu^2}{4}\right)^{1/3} \left (1+W\left ( b \, e^{12 t/\kappa+b} \right ) \right )^{1/3} \, , & b \neq 0 \\  \\ 1 \, , \quad & b=0\end{matrix} \right.
\end{equation}
where we have used the Lambert W function (introduced in Remark \ref{rem:ne}). It can be easily checked that for $b \geq 0$ the corresponding solution $\sigma_t$ exists for all times (for $b$ strictly positive, it tends to $\left ( \frac{\kappa \mu^2}{4}\right)^{1/3}$  as $t \rightarrow -\infty$ and grows indefinitely as $t \rightarrow +\infty$), while for $b<0$ the solution is only defined in the interval $(-\infty, t_\ast)$, where $t_\ast=-\kappa/12 (1+b+\log(-b))$. In this latter case, the solution also asymptotes to $\left ( \frac{\kappa \mu^2}{4}\right)^{1/3}$ as $t \rightarrow -\infty$.

\noindent
On the other hand, if $\mu=0$ the only solution corresponds to the static one, namely $\sigma_t=1$, whereas for $\kappa=0$ the solution reads $\sigma_t=(1+3t \mu^2)^{1/3}$. The latter is defined for $t \in (-(3\mu^2)^{-1}, +\infty)$.

\begin{figure}[ht]
    \centering
    \includegraphics[scale=0.33]{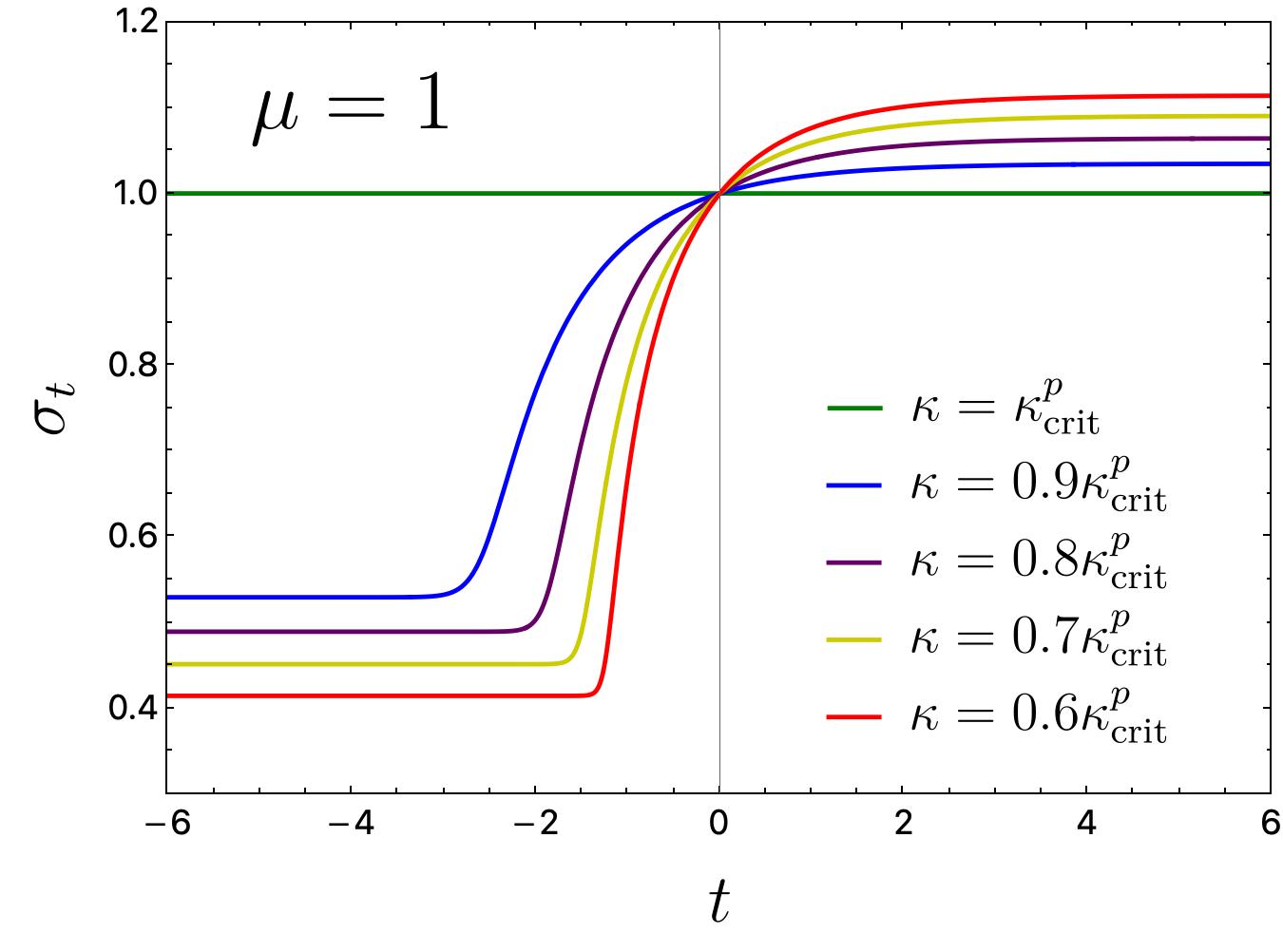}
    \includegraphics[scale=0.33]{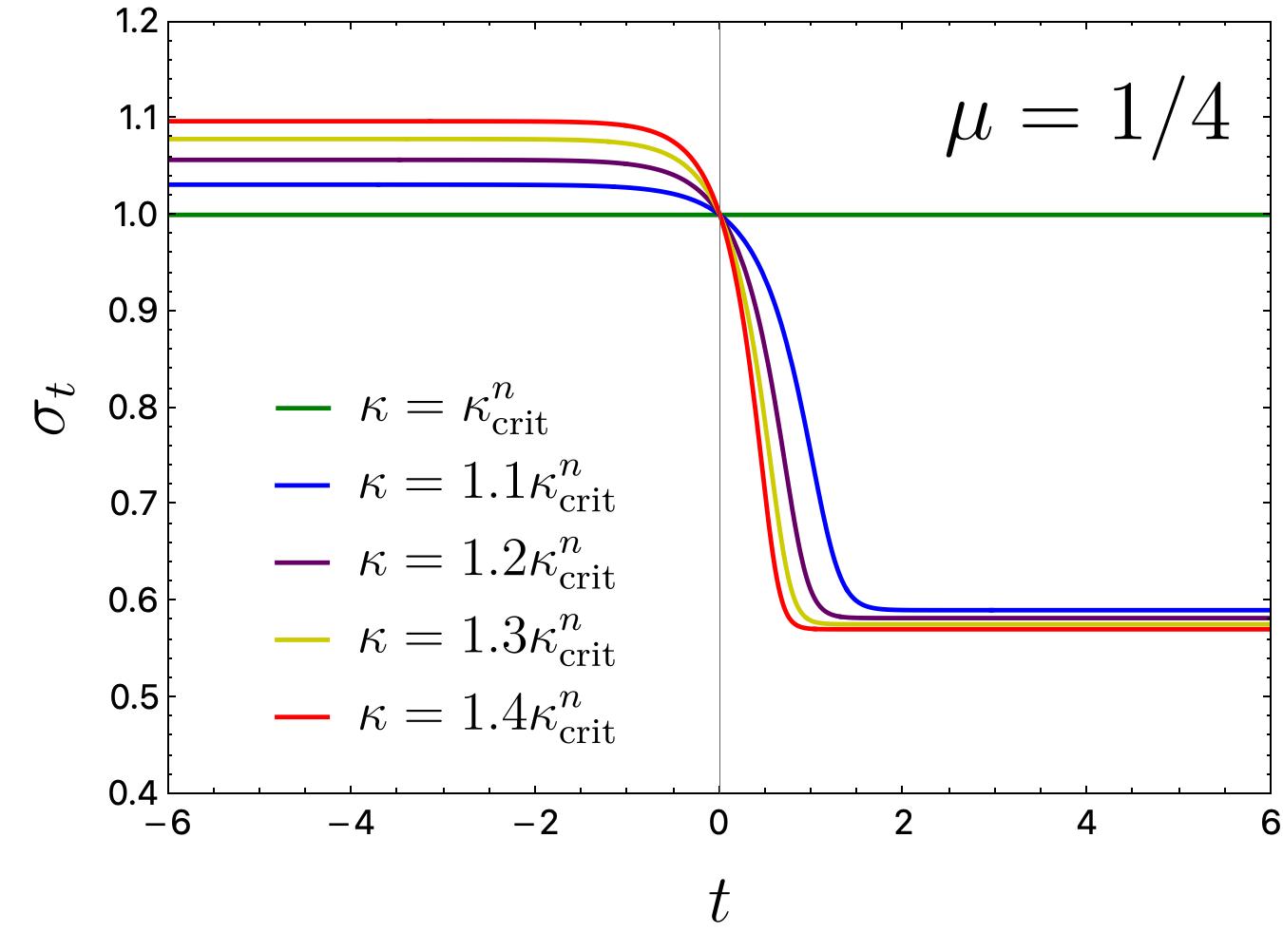}
    \caption{Left: numerical solutions of the positive homothety flow defined by Equation \eqref{eq:sigmatpositive}, for the specific values of $\kappa$ and $\mu$ indicated in the graph. Right: numerical solutions of the negative homothety flow defined by Equation \eqref{eq:sigmatnegative}, for the specific values of $\kappa$ and $\mu$ indicated in the graph. In both cases, the solutions correspond to eternal and regular flows.}
    \label{fig:graphs}
\end{figure}

\subsubsection{Negative homothety flows}

In analogy with the positive case, it is not possible to find explicit solutions for generic $(\kappa,\mu)$, so we restrict ourselves to a qualitative and numerical study of the solutions of Equation \eqref{eq:sigmatnegative}. We define:
\begin{equation*}
    F_n(\kappa,\mu,y)=2\,\left (\frac{\kappa}{9} + \frac{1}{3} \right ) y - \frac{\kappa}{3} + \frac{\mu^2}{y} + \frac{\kappa\,\mu^2}{y^2} - \frac{\kappa \,\mu^4}{4 y^4}    \, .
\end{equation*}
We start by examining the existence of static solutions. Demanding  $F_n(\kappa,\mu,1)=0$, we find:
\begin{equation*}
    \kappa(\mu)=\kappa_{\mathrm{crit}}^n(\mu)=\frac{36 \mu^2+24}{9\mu^2(\mu^2-4)+4}\,.
\end{equation*}
We note that for $\mu_-<\mu <\mu_+$, with $\mu_{\pm}^2=2/3(3\pm 2 \sqrt{2})$, there are no static solutions, since $\kappa_{\mathrm{crit}}^n(\mu)$ would be negative. This suggests to split our study depending on the value of $\mu^2$:
\begin{itemize}
    \item If $\mu_-^2>\mu^2>0$, there are two different subcases to consider:
    \begin{itemize}
        \item If $\kappa<\kappa_{\mathrm{crit}}^n(\mu)$, we have that $F_n(\kappa,\mu,1)>0$, with $F_n(\kappa,\mu,y)$ growing monotonically as we increase $y$ and diverging as $y \rightarrow +\infty$. Also, we observe that there exists a value $y_0 \in (0,1)$ such that $F_n(\kappa,\mu,y_0)=0$. Therefore, solutions $\sigma_t$ in this case are eternal and tend to a finite and positive value as $t \rightarrow -\infty$ and diverge as $t \rightarrow +\infty$. The case $\kappa=\kappa_{\mathrm{crit}}^n(\mu)$ corresponds to the static solution $\sigma_t=1$.
        \item If $\kappa>\kappa_{\mathrm{crit}}^n(\mu)$, then $F_n(\kappa,\mu,1)<0$. Furthermore, it is verified that there always exist $y_1 \in (0,1)$ and $y_2>1$ such that $F_n(\kappa,\mu,y_1)=F_n(\kappa,\mu,y_2)=0$. Consequently, these solutions are eternal and regular flows satisfying that $+\infty> \lim_{t \rightarrow -\infty} \sigma_t>1$ and $1 > \lim_{t \rightarrow +\infty} \sigma_t>0$. We show some instances of these solutions in Figure \ref{fig:graphs} (right).
    \end{itemize}
    \item If $\mu_+^2\geq \mu^2 \geq \mu_-^2$, regardless of the value of $\kappa$ we observe the same qualitative behaviour: the subsequent solution $\sigma_t$ exists for all times, asymptoting to a finite value as $t \rightarrow -\infty$ and diverging as $t\rightarrow +\infty$. This is due to $F_n(\kappa,\mu,y) >0$ for every $y\geq 1$ and to the existence of $y_0 \in (0,1)$ such that $F_n(\kappa,\mu,y_0)=0$.
    \item Finally, for $\mu^2>\mu_+^2$, two subcases need to be addressed:
    \begin{itemize}
        \item For $\kappa<\kappa_{\mathrm{crit}}(\mu)$, we have that $F_n(\kappa,\mu,y)>0$ for $y \geq 1$, always existing $y_0 \in (0,1)$ such that $F_n(\kappa,\mu,y_0)=0$. Therefore, the subsequent solutions $\sigma_t$ are defined for all times, remaining finite as $t \rightarrow -\infty$ and diverging when $t \rightarrow +\infty$. The case $\kappa=\kappa_{\mathrm{crit}}(\mu)$ corresponds to the static solution $\sigma_t=1$.
        \item If $\kappa>\kappa_{\mathrm{crit}}(\mu)$, it is easily verified that $F_n(\kappa,\mu,1)<0$. Since $F_n(\kappa_\mu,y)$ decreases monotonically as $y$ takes smaller values and diverges when $y \rightarrow 0^+$, we conclude that these solutions collapse at finite time. 
    \end{itemize}
\end{itemize}
For $\kappa=0$, the flow collapses at finite time regardless the value of $\mu \neq 0$ (if $\mu=0$, then $\sigma_t=2t/3+1$). Finally, if $\mu=0$, it is observed that the solution collapses at finite time if $\kappa>6$, it corresponds to the static one for $\kappa=6$ and it yields a solution that asymptotes to a finite and positive number for $t \rightarrow -\infty$ and diverges as $t \rightarrow +\infty$ for $\kappa<6$.

%Whenever $\kappa \neq \kappa_{\mathrm{crit}}(\mu)$, the solution is no longer static. We observe that:
%\begin{enumerate}
    %\item  For $\kappa < \kappa_{\mathrm{crit}}(\mu)$, we have checked numerically that all solutions  $\sigma_t$ are well-defined for all times $t \in \mathbb{R}$, satisfying that $1> \lim_{t \rightarrow -\infty} \sigma_t>0$ and $+\infty > \lim_{t \rightarrow +\infty} \sigma_t>1$. 
   % \item On the other hand, if $\kappa > \kappa_{\mathrm{crit}}(\mu)$, the corresponding solution either exists for all times or turns out to diverge at finite time.
%\end{enumerate}

%%%%%%%%%%%%%%%%%%%%%%%%%%%%%%%%%%%%%%%%%%%%%%%%%%%%%%%%%%%%%%%%%%%%%%%%%%%%

\section{Heterotic supergravity with trivial gauge bundle}
\label{sec:HS}

%%%%%%%%%%%%%%%%%%%%%%%%%%%%%%%%%%%%%%%%%%%%%%%%%%%%%%%%%%%%%%%%%%%%%%%%%%%%

The notion of \emph{Heterotic soliton} was introduced in \cite{Moroianu:2021kit} as a manifold equipped with a solution to the bosonic sector of Heterotic supergravity with trivial gauge bundle at first order in the string slope parameter, which we called the \emph{Heterotic soliton system}. In the following, we revisit the Heterotic soliton system and discuss some of its basic geometric properties.

%%%%%%%%%%%%%%%%%%%%%%%%%%%%%%%%%%%%%%%%%%%%%%%%%%%%%%%%%%%%%%%%%%%%%%%%%%%%

\subsection{The Heterotic soliton system}

%%%%%%%%%%%%%%%%%%%%%%%%%%%%%%%%%%%%%%%%%%%%%%%%%%%%%%%%%%%%%%%%%%%%%%%%%%%%

Heterotic solitons are expected to be solitons for the Heterotic-Ricci flow, an expectation that we will prove to be correct in the three-dimensional case considered in Section \ref{sec:solitonsHR}, and hence are expected to define a subclass of Heterotic-Ricci solitons for which there exists an interpretation within the context of supergravity.  

\begin{definition}
\label{def:motionsugratorsion}
Let $\kappa \geq 0$ be a non-negative real constant. The \emph{Heterotic soliton system} on a manifold $M$ is the following system of partial differential equations:
\begin{eqnarray}
\label{eq:motionsugratorsion}	
\mathrm{Ric}^{g,H} +  \nabla^{g,H}\varphi  + \kappa\, \mR^{g,H} \circ_g\mR^{g,H} = 0\, , \quad \delta^g\varphi + |\varphi|^2_g  - |H|^2_g  +  \kappa\, |\mR^{g,H}|^2_{g} = 0
\end{eqnarray}
	
\noindent
together with the \emph{Bianchi identity}:
\begin{equation}
\label{eq:BianchiIdentity}
\dd H + \kappa \langle \mR^{g,H}\wedge \mR^{g,H}\rangle_g = 0
\end{equation}
	
\noindent
for triples $(g,\varphi,H)$, where $g$ is a Riemannian metric on $M$, $\varphi\in \Omega^1(M)$ is a closed one-form and $H\in \Omega^3(M)$ is a three-form. A \emph{Heterotic soliton} is a triple $(g,\varphi,H)$ satisfying the Heterotic soliton system.
\end{definition}

\begin{remark}
Evidently, the Bianchi identity \eqref{eq:BianchiIdentity} is not an \emph{identity} but an equation that needs to be solved. The terminology comes from the physics community and is nowadays pervasive also in the mathematics community.
\end{remark}

\noindent
We will denote the configuration space and solution space of the Heterotic soliton system by $\Conf(M)$ and $\Sol_{\kappa}(M)$, respectively. Given $(g,\varphi,H)\in \Conf(M)$, the cohomology class $[\varphi]\in H^1(M,\mathbb{R})$ determined by $\varphi$ will be called the \emph{Lee class} of the configuration. This terminology is inherited from \cite{Garcia-Fernandez:2018emx} in the context of the \emph{twisted} Hull-Strominger system, for which $\varphi$ is necessarily the Lee form of a Hermitian structure on a complex manifold. By the explicit form of Equations \eqref{eq:motionsugratorsion} and \eqref{eq:BianchiIdentity}, which are completely determined by supersymmetry, it is clear that the Heterotic soliton system can be considered to be a natural differential system for Riemannian geometries with torsion.  

\begin{remark}
The Heterotic soliton system corresponds to the equations of motion of bosonic Heterotic supergravity with trivial gauge bundle at first order in the string slope parameter $\kappa$ \cite{BRI,BRII}. The symmetric part of the first equation in \eqref{eq:motionsugratorsion} is usually referred to in the literature as the \emph{Einstein equation}, whereas its skew-symmetric part is usually called the \emph{Maxwell equation}. On the other hand, the second equation in \eqref{eq:motionsugratorsion} is commonly referred to in the literature as the \emph{dilaton equation}.  
\end{remark}

\noindent
A Heterotic soliton $(g,\varphi,H)\in \Sol_{\kappa}(M)$ for which both $\varphi$ and $H$ vanish identically reduces to a Ricci-flat metric. Hence, we will refer to such Heterotic solitons as \emph{Ricci-flat}. If $g$ is, in addition, flat, we will say that such a soliton is \emph{trivial}. A soliton $(g,\varphi,H)\in \Sol_{0}(M)$ with $\kappa=0$ is a particular instance of steady generalized Ricci soliton \cite{FernandezStreetsLibro} which is Ricci-flat if and only if $\varphi$ is an exact one-form. Hence, non-Ricci-flat Heterotic solitons with $\kappa=0$ are particular cases of non-gradient steady generalized Ricci solitons. In the following, we will always assume that $\kappa>0$. For further reference, we introduce the following maps based on the defining equations of the Heterotic soliton system:
\begin{align}
& \nonumber \cE_{\mathrm{E}}  \colon \Conf(M) \to \Gamma(T^{\ast}M\otimes T^{\ast}M)\, ,\quad (g,\varphi,H) \mapsto \mathrm{Ric}^{g,H} +  \nabla^{g, H}\varphi  + \kappa\, \mR^{g,H} \circ_g\mR^{g,H}\\
& \nonumber \cE_{\mathrm{D}}  \colon \Conf(M) \to \cC^{\infty}(M)\, ,\quad (g,\varphi,H) \mapsto \nabla^{g\ast}\varphi + |\varphi|^2_g  - |H|^2_g  +  \kappa\, |\mR^{g,H}|^2_{g} \\
&\cE_{\mathrm{B}}  \colon \Conf(M) \to \Omega^4(M)\, , \quad (g,\varphi,H) \mapsto \dd H + \kappa \langle \mR^{g,H}\wedge \mR^{g,H}\rangle_g \label{eq:eqsmap}
\end{align}

\noindent
which we interpret as smooth maps of Fréchet manifolds. Clearly, the solution space $\Sol_{\kappa}(M)$ of the system corresponds to the preimage of zero by $\cE\colon = (\cE_{\mathrm{E}},\cE_{\mathrm{D}},\cE_{\mathrm{B}})$. We incidentally note that very little is known about the structure and properties $\Sol_{\kappa}(M)$, our main conjecture being that it should be a finite-dimensional smooth manifold locally around an irreducible solution $(g,\varphi,H)\in \Sol_{\kappa}(M)$, namely a solution of the system with no automorphisms.

\begin{remark}
For every $(g,\varphi,H)\in \Conf(M)$ the symmetric and skew-symmetric projections $\cE_{\mathrm{E}}^s(g,\varphi,H)\in \Gamma(T^{\ast}M^{\odot 2})$ and $\cE_{\mathrm{E}}^a(g,\varphi,H)\in \Omega^2(M)$ of $\cE_{\mathrm{E}}(g,\varphi,H)$ are explicitly given by:
\begin{equation*}
\cE_{\mathrm{E}}^s(g,\varphi,H) = \mathrm{Ric}^{g} +  \nabla^{g}\varphi - \frac{1}{2} H \circ_g H  + \kappa\, \mR^{g,H} \circ_g\mR^{g,H}\, , \quad \cE_{\mathrm{E}}^a(g,\varphi,H)= \frac{1}{2}\delta^{g} H + \frac{1}{2} H(\varphi) \, .
\end{equation*}  

\noindent
These formulas can be very convenient to do explicit computations.
\end{remark}

\noindent
Due to the fact that we are studying an exact truncation of Heterotic supergravity at first order in the string slope parameter, the divergence of the Einstein equation $\cE_{\mathrm{E}}^s$ evaluated on an element $(g,\varphi,H)$ satisfying the remaining equations of the system may not vanish identically, as we proceed to show below through a calculation that will naturally lead us to introduce the notion of \emph{strong} Heterotic soliton, see Definition \ref{def:strongsystem}.

\begin{lemma}
\label{lemma:divergenceRHRH}
The following identity holds:
\begin{equation*}
\nabla^{g\ast} (\mR^{g,H}\circ_g \mR^{g,H})(v) = \langle \nabla^{g,H\ast}\mR^{g,H} ,\mR^{g,H}_v \rangle_g - \frac{1}{2} v (\vert \mR^{g,H}\vert^2_g) - \frac{1}{2} \langle H , v\lrcorner \langle \mR^{g,H}\wedge \mR^{g,H}\rangle_g\rangle_g
\end{equation*}
	
\noindent
for every $v\in \mathfrak{X}(M)$, where $\nabla^{g,H\ast}\colon \Gamma(\wedge^2 M\otimes \wedge^2 M) \to \Gamma(\wedge^1 M\otimes \wedge^2M)$ is the formal $L^2$ adjoint of $\dd^{\nabla^{g,H}}\colon \Gamma(\wedge^1 M\otimes \wedge^2 M) \to \Gamma(\wedge^2 M\otimes \wedge^2M)$ with respect to the determinant norm and $\nabla^{g\ast}\colon \Gamma(T^{\ast}M\otimes T^{\ast}M) \to \Omega^1(M)$ is the formal $L^2$ adjoint of $\nabla^g\colon \Omega^1(M) \to \Gamma(T^{\ast}M\otimes T^{\ast}M)$.
\end{lemma}

\begin{proof}
We choose an orthonormal frame $\left\{ e_i \right\}$ parallel with respect to $\nabla^g$ at a given point $m\in M$. For every $v\in\mathfrak{X}(M)$ parallel with respect to $\nabla^g$ at $m$ we compute (using from now on Einstein's summation convention on repeated indices):
\begin{eqnarray*}
\nabla^{g\ast} (\mR^{g,H}\circ_g \mR^{g,H})(v)  &=& -(\nabla^{g}_{e_i}(\mR^{g,H}\circ_g \mR^{g,H}))(e_i, v)=- \dd\langle \mR^{g,H}_{e_i, e_j} , \mR^{g,H}_{v,e_j}\rangle_g(e_i)\\& =& \langle \nabla^{g,H\ast}\mR^{g,H} , \mR^{g,H}_v \rangle_g - \langle \mR^{g,H}_{e_i,e_j} , (\nabla^{g,H}_{e_i}\mR^{g,H})_{v , e_j} \rangle_g\\ &=& \langle \nabla^{g,H\ast}\mR^{g,H} , \mR^{g,H}_v \rangle_g  - \frac{1}{2}\langle \mR^{g,H}_{e_i , e_j}, \cC(\nabla^{g,H}_{e_i}\mR^{g,H})_{v , e_j} + (\nabla^{g,H}_{v}\mR^{g,H})_{e_i , e_j}  \rangle_g \\&=& \langle \nabla^{g,H\ast}\mR^{g,H} , \mR^{g,H}_v \rangle_g - \frac{1}{2}\langle \mR^{g,H}_{e_i , e_j} , \cC(\nabla^{g,H}_{e_i}\mR^{g,H})_{v , e_j}\rangle_g - \frac{1}{2} v(\vert \mR^{g,H}\vert^2_g)
\end{eqnarray*}
	
\noindent
where $\cC$ denotes here the sum over all cyclic permutations of the set $(e_i,v,e_j)$. Using now the second Bianchi identity for connections with torsion we obtain after some tedious but straightforward computation:
\begin{equation*}
\langle \mR^{g,H}_{e_i , e_j} , \cC(\nabla^{g,H}_{e_i}\mR^{g,H})_{v , e_j}\rangle_g = \langle H , v\lrcorner \langle \mR^{g,H}\wedge \mR^{g,H}\rangle_g\rangle_g
\end{equation*}

\noindent
and we conclude.
\end{proof}

\noindent
We will also need the following:

\begin{lemma}
\label{lemma:divHH}
The following identity holds for every configuration $(g,\varphi,H)\in \Conf(M)$:
\begin{eqnarray*}
& \nabla^{g\ast}(H\circ_g H)(v) = - \frac{1}{2} v(\vert H\vert_g^2) - (H\circ_g H)(\varphi,v) - \kappa \langle H , v\lrcorner \langle \mR^{g,H}\wedge \mR^{g,H}\rangle_g\rangle_g \\
& + 2 \langle \cE^a_{\mathrm{E}}(g,\varphi,H), H(v)\rangle_g + \langle H , v\lrcorner\, \cE_{\mathrm{B}}(g,H)\rangle_g
\end{eqnarray*}

\noindent
where $v\in \mathfrak{X}(M)$ is a vector field.
\end{lemma}

\begin{proof}
We choose an orthonormal frame $\left\{ e_i \right\}$ adapted to $\nabla^g$ at a given point $m\in M$ and compute:
\begin{eqnarray*}
& -\nabla^{g\ast}(H\circ_g H)(v) = (\nabla^g_{e_i} (H\circ_g H))(e_i,v) = \dd\langle H(e_i) , H(v)\rangle_g(e_i) = -\langle \nabla^{g\ast}H , H(v)\rangle_g \\
&   + \langle H(e_i), v\lrcorner \nabla^g_{e_i}H\rangle_g
 = \langle H(\varphi) , H(v)\rangle_g - 2 \langle \cE^a_E(g,\varphi,H), H(v)\rangle_g - \langle H , v\lrcorner (e_i\wedge \nabla^g_{e_i}H) \rangle_g + \langle H , \nabla^g_v H \rangle_g \\
& = \frac{1}{2} v(\vert H\vert_g^2) - 2 \langle \cE^a_{\mathrm{E}}(g,\varphi,H), H(v)\rangle_g +\langle H(\varphi),H(v)\rangle_g + \kappa \langle H , v\lrcorner \langle \mR^{g,H}\wedge \mR^{g,H}\rangle_g\rangle_g - \langle H , v\lrcorner \cE_{\mathrm{B}}(g,H)\rangle_g
\end{eqnarray*}

\noindent
where $v\in\mathfrak{X}(M)$ is any vector field adapted to $\nabla^g$ at $m\in M$.
\end{proof}

\begin{proposition}
The following formula holds:
\begin{eqnarray*}
& v\lrcorner( \nabla^{g\ast}\cE^s_{\mathrm{E}}(g,\varphi,H) + \varphi\lrcorner \cE^s_{\mathrm{E}}(g,\varphi,H) + \frac{1}{2} \dd \mathrm{Tr}_g(\cE^s_{\mathrm{E}}(g,\varphi,H)) )= \kappa \langle  \mR^{g,H}_v , \nabla^{g,H\ast}\mR^{g,H} + \varphi\lrcorner\mR^{g,H}\rangle_g \\ 
& + \frac{1}{2} v\lrcorner \dd \cE_{\mathrm{D}}(g,\varphi,H) - \langle \cE^a_{\mathrm{E}}(g,\varphi,H), H(v)\rangle_g  - \frac{1}{2} \langle H , v\lrcorner\, \cE_{\mathrm{B}}(g,H)\rangle_g
\end{eqnarray*}

\noindent
for every $(g,\varphi,H)\in \Conf(M)$ and $v \in \mathfrak{X}(M)$.  
\end{proposition}

\begin{proof}
Let $(g,\varphi,H)\in \Conf(M)$. First, recall the standard identities:
\begin{eqnarray*}
\nabla^{g\ast}\mathrm{Ric}^g = -\frac{1}{2} \dd s_g\, , \quad  \nabla^{g\ast}\nabla^g\varphi = \Delta_g\varphi - \mathrm{Ric}^g(\varphi)\, .
\end{eqnarray*}

\noindent
Using Lemmas \ref{lemma:divergenceRHRH} and \ref{lemma:divHH} together with the previous standard identities we obtain:
\begin{eqnarray*}
& v\lrcorner\nabla^{g\ast}\cE^s_{\mathrm{E}}(g,\varphi,H) = -\frac{1}{2} v (s_g) + \frac{1}{4}v(\vert H\vert^2) + \frac{1}{2} (H\circ_g H)(\varphi,v) + \Delta_g\varphi(v) - \mathrm{Ric}^g(\varphi,v) - \frac{\kappa}{2} v(\vert\mR^{g,H}\vert_g^2)\\
&   + \kappa \langle \mR^{g,H}_v , \nabla^{g,H\ast}\mR^{g,H} \rangle_g  -    \langle \cE^a_{\mathrm{E}}(g,\varphi,H), H(v)\rangle_g  - \frac{1}{2} \langle H , v\lrcorner\, \cE_{\mathrm{B}}(g,H)\rangle_g
\end{eqnarray*}

\noindent
for every vector field $v\in \mathfrak{X}(M)$. On the other hand, a quick computation reveals that:
\begin{eqnarray*}
v\lrcorner\varphi\lrcorner \cE^s_{\mathrm{E}}\,(g,\varphi,H) = \mathrm{Ric}^{g}(\varphi,v) +  \frac{1}{2} v(\vert\varphi\vert_g^2)  - \frac{1}{2} (H \circ_g H)(\varphi,v)  + \kappa\, (\mR^{g,H} \circ_g\mR^{g,H})(\varphi,v)
\end{eqnarray*}

\noindent
as well as:
\begin{equation*}
\dd \mathrm{Tr}_g(\cE^s_{\mathrm{E}}(g,\varphi,H)) = \dd ( s_g - \frac{3}{2} \vert H\vert_g^2 - \delta^g\varphi + 2 \kappa \vert \mR^{g,H}\vert_g^2)\, .
\end{equation*}

\noindent
Hence:
\begin{eqnarray*}
& v\lrcorner(\nabla^{g\ast}\cE^s_{\mathrm{E}}(g,\varphi,H) + \varphi\lrcorner\, \cE^s_{\mathrm{E}}(g,\varphi,H) + \frac{1}{2} \dd \mathrm{Tr}_g(\cE^s_{\mathrm{E}}(g,\varphi,H)) )=   \kappa \langle \mR^{g,H}_v , \nabla^{g,H\ast}\mR^{g,H} \rangle_g   \\
&  + \kappa\, (\mR^{g,H} \circ_g\mR^{g,H})(\varphi,v)  +   \frac{1}{2} \dd ( \delta^g\varphi + \vert\varphi\vert_g^2 -  \vert H\vert_g^2  + \kappa \vert \mR^{g,H}\vert_g^2)(v)  -    \langle \cE^a_{\mathrm{E}}(g,\varphi,H), H(v)\rangle_g   \\
&- \frac{1}{2} \langle H , v\lrcorner\, \cE_{\mathrm{B}}(g,H)\rangle_g =  \kappa \langle \mR^{g,H}_v , \nabla^{g,H\ast}\mR^{g,H} \rangle_g + \kappa \langle \mR^{g,H}_v , \mR^{g,H}_{\varphi}\rangle_g  + \frac{1}{2} v\lrcorner \dd \cE_{\mathrm{D}}(g,\varphi,H) \\
& -    \langle \cE^a_{\mathrm{E}}(g,\varphi,H), H(v)\rangle_g  - \frac{1}{2} \langle H , v\lrcorner\, \cE_{\mathrm{B}}(g,H)\rangle_g
\end{eqnarray*}

\noindent
which gives the equation in the statement.
\end{proof}

\noindent
Clearly, by the previous proposition, every Heterotic soliton satisfies:
\begin{equation*}
\langle  \mR^{g,H}_v , \nabla^{g,H\ast}\mR^{g,H} + \varphi\lrcorner\mR^{g,H}\rangle_g = 0
\end{equation*}

\noindent
for every vector field $v\in\mathfrak{X}(M)$. This motivates the following definition.

\begin{definition}
\label{def:strongsystem}
The \emph{strong Heterotic soliton system} is the Heterotic soliton system for triples $(g,\varphi,H)\in \Conf(M)$ together with the \emph{strong condition}:
\begin{equation}
\label{eq:strongcondition}
\nabla^{g,H\ast}\mR^{g,H} + \varphi\lrcorner\mR^{g,H} = 0\, .
\end{equation}

\noindent
A \emph{strong Heterotic soliton} is a solution of the strong Heterotic soliton system.
\end{definition}

\noindent
The strong Heterotic soliton system is an overdetermined system of partial differential equations. Despite being overdetermined, as we will see below it still has natural non-trivial solutions. This is a remarkable fact that can be traced back to the internal consistency  and coherence of supersymmetric field theories and, in particular, Heterotic supergravity and its various truncations. Discussing the consistency of the perturbative expansion Heterotic supergravity and its consistent truncations is beyond the scope of this article. The interested reader is referred to \cite{Melnikov:2014ywa} and its references for more details. 

\begin{remark}
\label{remark:existenceaction}
Note that every $(g,\varphi,H)\in \Conf(M)$ satisfying $\cE^{a}_{\mathrm{E}}(g,\varphi,H)=0$, $\cE_{\mathrm{D}}(g,\varphi,H)=0$ and $\cE_{\mathrm{B}}(g,\varphi,H)=0$ together with the strong condition automatically satisfies:
\begin{equation*}
\nabla^{g\ast}\cE_{\mathrm{E}}(g,\varphi,H) + \varphi\lrcorner \,\cE_{\mathrm{E}}(g,\varphi,H) + \frac{1}{2} \dd \mathrm{Tr}_g(\cE_{\mathrm{E}}(g,\varphi,H)) = 0\, .
\end{equation*}

\noindent
Identities of this type, involving the divergence of the Einstein equation, are closely related to the existence of variational principles associated with the Heterotic soliton system. Studying the action functionals that can be associated with the Heterotic soliton system is a fundamental problem that is currently a work in progress and that we will consider elsewhere.
\end{remark}

\noindent
The strong condition is an equation in $\wedge^1 M\otimes \wedge^2 M$. Hence, by projecting to $\wedge^3 M$ we obtain an equation in $\wedge^3 M$ that is satisfied by any strong Heterotic soliton and does not explicitly involve any curvature operators. On the other hand, by contraction we obtain an equation on $\wedge^1 M$.

\begin{proposition}
Let $(g,\varphi,H)\in \Conf(M)$ be a solution to the strong condition. Then, the following equations hold:
\begin{eqnarray}
&\nonumber \nabla^{g\ast}\nabla^g H + \frac{1}{2}\delta^g \dd H - \frac{1}{2} (\delta^g H)_{e_i}\wedge H_{e_i} + \frac{1}{2} H_{e_j}(e_i)\wedge (\nabla^g_{e_j}H)_{e_i} + \frac{1}{2} H_{e_j}(\nabla^g_{e_j} H) + \frac{1}{4} e_j \lrcorner\, H_{e_j}(\dd H) \\
&  - \frac{1}{4}  H_{e_j}(H_{e_j}(e_i)\wedge H_{e_i}) + \nabla^g_{\varphi}H + \frac{1}{2}\varphi\lrcorner\dd H - \frac{1}{2} H_{\varphi}(e_i)\wedge H_{e_i} = 0 \label{eq:Lambda3strong}
\end{eqnarray}

\noindent
where $(e_i)$ is any local orthonormal frame.
\end{proposition}

\begin{proof}
Let $v$ be any tangent vector. By \eqref{rgh} and the first Bianchi identity for $\mR^g$, the totally skew-symmetric part of $\mR^{g,H}_{v}$ is explicitly given by:
\begin{eqnarray*}
& e_i\wedge\mR^{g,H}_{v \, e_i} = - \frac{1}{2} e_i \wedge(\nabla^g_{v} H)(e_i) + \frac{1}{2} e_i\wedge (\nabla^g_{e_i} H)(v) + \frac{1}{4} e_i\wedge [H_{v} , H_{e_i}] \\
& = -\nabla^g_{v} H - \frac{1}{2} v\lrcorner\, \dd H + \frac{1}{2} H_{v}(e_i)\wedge H_{e_i}\, .
\end{eqnarray*}

\noindent
We choose the local orthonormal frame $(e_i)$ to be $\nabla^g$-parallel at some point $m\in M$. Taking $v=e_j$ in the above formula, differentiating with respect to $e_j$ and summing over $j$, the skew-symmetric part of $\nabla^{g,H\ast}\mR^{g,H}$ at $m$ can be computed as follows::
\begin{eqnarray*}
 e_i\wedge  (\nabla^{g,H}_{e_j} \mR^{g,H})_{e_j e_i} & =&  (\nabla^g_{e_j} - \frac{1}{2} H_{e_j}) (-\nabla^g_{e_j} H - \frac{1}{2} e_j \lrcorner\, \dd H + \frac{1}{2} H_{e_j}(e_i)\wedge H_{e_i})\\
& =&\nabla^{g\ast}\nabla^g H + \frac{1}{2}\delta^g \dd H - \frac{1}{2} (\delta^g H)_{e_i}\wedge H_{e_i} + \frac{1}{2} H_{e_j}(e_i)\wedge (\nabla^g_{e_j}H)_{e_i}\\
& &+\frac{1}{2} H_{e_j}(\nabla^g_{e_j} H) + \frac{1}{4} e_j \lrcorner\, H_{e_j}(\dd H) - \frac{1}{4}  H_{e_j}(H_{e_j}(e_i)\wedge H_{e_i})
\end{eqnarray*}

\noindent
where we have used that: 
\begin{equation*}
  H_{e_j}(e_j \lrcorner\,\dd H) = e_j \lrcorner\, H_{e_j}(\dd H)\, .
\end{equation*}

\noindent
Combining the previous equations together with the skew-symmetrization of \eqref{eq:strongcondition}, yields \eqref{eq:Lambda3strong}.
\end{proof}

\begin{remark}
Equation \eqref{eq:Lambda3strong} will be fundamental for the classification of the three-dimensional strong Heterotic solitons, considered in subsection \ref{sec:3dstrongsolitons}, where it becomes a scalar equation, see Lemma \ref{lemma:strongfconstant}.
\end{remark}

%%%%%%%%%%%%%%%%%%%%%%%%%%%%%%%%%%%%%%%%%%%%%%%%%%%%%%%%%%%%%%%%%%%%%%%%%%%%

\subsection{Lift to the frame bundle}

%%%%%%%%%%%%%%%%%%%%%%%%%%%%%%%%%%%%%%%%%%%%%%%%%%%%%%%%%%%%%%%%%%%%%%%%%%%%

We have introduced the Heterotic soliton system as a system of differential equations for a metric connection with skew-symmetric torsion that is determined by Heterotic supergravity and implicitly by the expected structure of the self-similar solutions to the Heterotic-Ricci flow. In the following, we show that the Heterotic soliton system, and especially its strong condition, can be naturally interpreted when \emph{lifted} to the reference frame bundle of the underlying manifold $M$. For this, we will adapt the construction of \cite[Proposition 7.1]{Baraglia:2013wua} to our case in which the \emph{gauge bundle} is not an abstract bundle and the \emph{gauge connection} is not an arbitrary connection but metric-compatible for each choice of Riemannian metric. 
 
Given $(g,\varphi,H)\in \Conf(M)$, we denote by $\mathrm{Fr}_g(M)$ the bundle of oriented frames of $(M,g)$. Given any metric connection $\nabla$ on $(M,g)$, in particular $\nabla^{g,H}$, we denote by:
\begin{equation*}
\mathrm{CS}(\nabla) \in \Omega^3(\mathrm{Fr}_g(M))
\end{equation*}
\noindent
its Chern-Simons form, which defines a three-form on the total space of $\mathrm{Fr}_g(M)$.  There is a natural map:
\begin{equation*}
\Conf(M) \to \Met(\mathrm{Fr}_g(M))\, , \qquad (g,\varphi,H)\mapsto \bar{g}_{\kappa}
\end{equation*}

\noindent
from $\Conf(M)$ into the Riemannian metrics over $\mathrm{Fr}_g(M)$ that, to every $(g,\varphi,H)$, associates the metric:
\begin{equation*}
\bar{g}_{\kappa}(X_1,X_2) = g_m(\dd\pi(X_1),\dd\pi(X_2)) - \kappa\, \mathrm{Tr}(\cA_{g,H}(X_1), \cA_{g,H}(X_2))\, , \qquad X_1, X_2 \in T_p\mathrm{Fr}_g(M)
\end{equation*}
\noindent
where $\cA_{g,H}$ is the metric connection $\nabla^{g,H}$, understood as a one-form in the total space of $\mathrm{Fr}_g(M)$ with values in $\mathfrak{so}(n)$. In particular, for every $X\in T_p \mathrm{Fr}_g(M)$ we have $\cA_{g,H}(X)\in \mathfrak{so}(n)$. Note that the metric $\bar{g}_{\kappa}$ is invariant under the principal bundle action of $\mathrm{Fr}_g(M)$.

\begin{proposition}\cite[Proposition 7.1]{Baraglia:2013wua}
A triple $(g,\varphi,H)\in \Conf(M)$ is a strong Heterotic soliton with constant $\kappa$ if and only if the associated triple $(\bar{g}_{\kappa},\bar{\varphi},\bar{H})$ satisfies:
\begin{eqnarray}
\label{eq:liftedequations}
\mathrm{Ric}^{\bar{g}_{\kappa},\bar{H}_{\kappa}} +  \nabla^{\bar{g}_{\kappa},\bar{H}_{\kappa}}\bar{\varphi} = 0\, , \qquad \delta^{\bar{g}_{\kappa}}\bar{\varphi} + |\bar{\varphi}|^2_{\bar{g}_{\kappa}}  - |\bar{H}_{\kappa}|^2_{\bar{g}_{\kappa}} = 0\, ,\qquad  \dd \bar{H}_{\kappa}=0
\end{eqnarray}

\noindent
where $\bar{\varphi} = \pi^{\ast}\varphi$ and  $\bar{H}_{\kappa} = H + \kappa\, \mathrm{CS}(\mathcal{A}_{g,H})$.
\end{proposition}

\begin{remark}
The interest of the previous proposition is that equations \eqref{eq:liftedequations} precisely correspond to the equations of motion of NS-NS supergravity for a pseudo-Riemannian metric of split signature and a three-form \emph{flux} $\bar{H}_k$ belonging to a given \emph{string class} \cite{Redden} for each choice of metric.  The fact that the signature of the metric $\bar{g}$ on the total space of $\mathrm{Fr}_g(M)$ is of split signature is not accidental and can be traced back to the sign of the higher order curvature term $\mR^{g,H}\circ_g \mR^{g,H}$ in \eqref{eq:motionsugratorsion}, see also Remark \ref{remark:positivenegative}. Had this term appeared with the opposite sign (recall that $\kappa$ is assumed to be non-negative), then the corresponding $\bar{g}$ would be positive definite. This would be however inconsistent with the physical interpretation of the system. 
\end{remark}

%%%%%%%%%%%%%%%%%%%%%%%%%%%%%%%%%%%%%%%%%%%%%%%%%%%%%%%%%%%%%%%%%%%%%%%%%%%%

\section{Heterotic solitons on three-manifolds}
\label{sec:solitonsHR}

%%%%%%%%%%%%%%%%%%%%%%%%%%%%%%%%%%%%%%%%%%%%%%%%%%%%%%%%%%%%%%%%%%%%%%%%%%%%

In this section, we consider the Heterotic soliton system on a compact three-manifold with the goal of obtaining a classification result for three-dimensional strong Heterotic solitons. 

%%%%%%%%%%%%%%%%%%%%%%%%%%%%%%%%%%%%%%%%%%%%%%%%%%%%%%%%%%%%%%%%%%%%%%%%%%%%

\subsection{Three-dimensional reformulation}

%%%%%%%%%%%%%%%%%%%%%%%%%%%%%%%%%%%%%%%%%%%%%%%%%%%%%%%%%%%%%%%%%%%%%%%%%%%%

The Heterotic soliton system can be simplified in three dimensions using the fact that the Hodge dual of $H$ is a function and the fact that the Riemann tensor is completely determined by the Ricci curvature. Given a triple $(g,\varphi,H)$, we write $H= f \nu_g$ in terms of the unique function $f = \ast_g H\in C^{\infty}(M)$, where $\nu_g$ denotes the Riemannian volume form associated to $g$. Consequently, we define the configuration space $\Conf(M)$ of the Heterotic soliton system on $M$ as the set of triples $(g,\varphi,f)$, where $g$ is a Riemannian metric, $\varphi$ is a one-form and $f$ is a function. Similarly, we define the set of Heterotic solitons $\Sol_{\kappa}(M)$ on $M$ as the set of triples $(g,\varphi,f)\in\Conf(M)$ such that $(g,\varphi,H = \ast_g f)$ satisfies the Heterotic soliton system with parameter $\kappa\in \mathbb{R}$. 

Using the Lemma \ref{lemma:vg} in Appendix \ref{app:curvature3d} we can write the Heterotic system in three dimensions as a system of differential equations for a metric $g$, a closed one-form $\varphi$ and a function $f$ involving the Ricci tensor and scalar curvature of $g$ as the only curvature operators.
 
\begin{proposition}
A triple $(g,\varphi, f)\in \Conf(M)$ is a Heterotic soliton if and only if it satisfies the following system of equations:
\begin{eqnarray}
\label{eq:motionHetsugra3d1}
& \nonumber -  \kappa \, \mathrm{Ric}^g \circ_g \mathrm{Ric}^g + (1 + \kappa s_g - \frac{\kappa}{2} f^2  )\mathrm{Ric}^{g} + \left ( \kappa  \vert \mathrm{Ric}^g \vert_g^2-\frac{\kappa}{2} s_g^2+\frac{\kappa}{4} \vert \dd f  \vert_g^2  - \frac{1}{2} f^2+\frac{\kappa}{8}f^4\right)  g \\ 
& + \frac{\kappa}{2}\, [\ast_g\dd f, \mathrm{Ric}^g] + \frac{\kappa }{4} \dd f \otimes \dd f + \nabla^{g}\varphi= 0\, ,\\
\label{eq:motionHetsugra3d2}
& f \varphi=\dd f\, , \qquad s_g= 3\, \delta^g \varphi+2 \vert \varphi\vert^2 -\frac{1}{2} f^2\, ,
\end{eqnarray}

\noindent
where $f = \ast_g H$.
\end{proposition}

\begin{proof}
Since in three dimensions the Bianchi equation \eqref{eq:Bianchi} is automatically satisfied, a triple $(g,\varphi, f)\in \Conf(M)$ is a Heterotic soliton if and only if $(g,\varphi, H = f \nu_g)$ satisfies equations \eqref{eq:motionsugratorsion}. Plugging Equation \eqref{eq:lemmavg} into $\cE_{\mathrm{E}}^s(g,\varphi,H) = 0$ and using the identity:
\begin{equation*}
H\circ_g H(v_1,v_2) = g(v_1,v_2) f^2
\end{equation*}

\noindent
it follows that Equation \eqref{eq:motionHetsugra3d1} is equivalent to $\cE_{\mathrm{E}}^s(g,\varphi,H) = 0$, namely it is equivalent to the symmetric part of the first equation in \eqref{eq:motionsugratorsion}. For the skew-symmetric part $\cE_{\mathrm{E}}^a(g,\varphi,H) = 0$ of $\cE_{\mathrm{E}}(g,\varphi,H)=0$ we compute:
\begin{equation*}
\delta^g H +  \iota_{\varphi} H =-\ast_g \dd \ast_g (f \nu_g) +  f \iota_{\varphi} \nu_g = -\ast_g \dd f +f \ast_g\varphi
\end{equation*}

\noindent
whence equation $\cE_{\mathrm{E}}^a(g,\varphi,H)=0$ is equivalent to the first equation in \eqref{eq:motionHetsugra3d2}. The second equation in \eqref{eq:motionsugratorsion} can be proven to be equivalent to the second equation in \eqref{eq:motionHetsugra3d2} by combining it appropriately to the trace of $\cE_{\mathrm{E}}(g,\varphi,H)=0$. 
\end{proof} 
 
\noindent
Therefore, the Heterotic soliton system, as introduced in Definition \ref{def:HeteroticSystem}, is equivalent in three dimensions to  equations \eqref{eq:motionHetsugra3d1} and \eqref{eq:motionHetsugra3d2} for tuples of the form $(g,\varphi,f)\in \Conf(M)$. 

\begin{definition}
\label{def:HeteroticSystem}
Let $\kappa>0$ be a positive real number. Equations \eqref{eq:motionHetsugra3d1} and \eqref{eq:motionHetsugra3d2} are the {\em three-dimensional Heterotic soliton system} on $M$. Solutions of this system are \emph{three-dimensional Heterotic solitons}.
\end{definition}

\noindent
The following result is the key reason why solutions to the three-dimensional Heterotic system are relevant in relation to the three-dimensional Heterotic-Ricci flow: the former are, as their name suggests, self-similar solutions of the latter. 

\begin{proposition}
Let $(g,\varphi,f)\in \Sol_{\kappa}(M)$ be a three-dimensional Heterotic soliton. Then, there exists a one-parameter family of diffeomorphisms $\left\{ \psi_t\right\}_{t\in \cI}$ for some interval $\cI$ such that:
\begin{equation*}
(g_t,f_t) = (\psi^{\ast}_t g , f\circ\psi_t)
\end{equation*}

\noindent
is a Heterotic-Ricci flow.
\end{proposition}

\begin{proof}
Suppose that $(g,\varphi,f)\in \Sol_{\kappa}(M)$ and define $\left\{ \psi_t\right\}_{t\in \cI}$ the one-parameter family of diffeomorphisms defined by the flow of $\varphi$. Then:
\begin{equation*}
\partial_t (\psi^{\ast}_t g) = \psi^{\ast}_t \mathcal{L}_{\varphi}g = 2\psi^{\ast}_t \nabla^{g}\varphi
\end{equation*}
\noindent
where we have used that $\varphi$ is closed. Plugging this equation into \eqref{eq:3dHRflowreformulated1} we directly obtain \eqref{eq:motionHetsugra3d1}. Regarding \eqref{eq:3dHRflowreformulated2}, we compute:
\begin{equation*}
\partial_t (f\circ \psi_t ) + \frac{1}{2} \mathrm{Tr}_{\psi^{\ast}_t g}(\partial_t \psi^{\ast}_t g) f\circ \psi_t  +  \Delta_{\psi^{\ast}_t g} (f\circ\psi_t) = \psi_t^{\ast} (\dd f(\varphi) + \mathrm{Tr}_g(\nabla^g\varphi) f + \delta^g\dd f) = 0\, .
\end{equation*}
\noindent
Here we have used that $\dd f = f\varphi$, which implies:
\begin{equation*}
\delta^g\dd f = -\dd f(\varphi) + f \delta^g\varphi
\end{equation*}
\noindent
together with the identity $\mathrm{Tr}_g(\nabla^g\varphi) = - \delta^g\varphi$. Hence $(g_t,f_t) = (\psi^{\ast}_t g , f\circ\psi_t)$ is a solution of the three-dimensional Heterotic-Ricci flow and consequently solutions of \eqref{eq:motionHetsugra3d1} and \eqref{eq:motionHetsugra3d2} are self-similar solutions of the Heterotic-Ricci flow.
\end{proof}

\begin{remark}
We expect Heterotic solitons in all dimensions to be actual solitons for the Heterotic-Ricci flow. However, in dimensions larger than four, $H_t$ is in general not closed, and therefore obtaining the solitons of the flow requires a precise understanding of the symmetries of the system when formulated either on a string structure \cite{Killingback,Redden,Waldorf} or or on a transitive Courant algebroid \cite{Baraglia:2013wua,Garcia-Fernandez:2016ofz}, see Remark \ref{remark:stringstructure}. We plan to come back to this issue in the future.
\end{remark}

\noindent
The strong condition introduced in Definition \ref{def:strongsystem} can, of course, be adapted to three dimensions, but in its generality, it reduces to an expression that is not particularly illuminating. We will come back to this point later, see subsection \ref{sec:3dstrongsolitons}.  

%%%%%%%%%%%%%%%%%%%%%%%%%%%%%%%%%%%%%%%%%%%%%%%%%%%%%%%%%%%%%%%%%%%%%%%%%%%%

\subsection{General properties}

%%%%%%%%%%%%%%%%%%%%%%%%%%%%%%%%%%%%%%%%%%%%%%%%%%%%%%%%%%%%%%%%%%%%%%%%%%%%

In this subsection, we will prove a number of structural results about three-dimensional Heterotic solitons.

\begin{proposition}
\label{prop:f0trivial}
Let $(M,g)$ be a closed Riemannian three-manifold. A three-dimensional Heterotic soliton is trivial if and only if $f=0$.
\end{proposition}

\begin{proof}
The \emph{only if} direction holds by definition. Therefore, assume that $f=0$. The trace of Equation \eqref{eq:motionHetsugra3d1} reduces to:
\begin{equation*}
\delta^{g}\varphi = 2\kappa \vert \mathrm{Ric}^{g}_0\vert^2_g +\frac{\kappa}{6} s_g^2 + s_g
\end{equation*}

\noindent
where $\mathrm{Ric}^{g}_0 = \mathrm{Ric}^g - s_g\, g/3$ is the trace-free part of the Ricci tensor. On the other hand, the second equation in \eqref{eq:motionHetsugra3d1} reduces to:
\begin{equation*}
s_g= 3\, \delta^{g} \varphi+2 \vert \varphi\vert^2_g\, ,
\end{equation*}

\noindent
Combining the previous equations and integrating we obtain:
\begin{equation*}
\int_M (2\kappa \vert \mathrm{Ric}^{g}_0\vert^2_g +\frac{\kappa}{6} s_g^2 + 2\vert \varphi\vert^2) \nu_g = 0
\end{equation*}

\noindent
where $\nu_g$ is the volume form of $g$. Hence $\varphi=s_g=\vert \mathrm{Ric}^{g}_0\vert^2_g=0$, so $g$ is Ricci-flat. Since $M$ is three-dimensional implies that $g$ is in fact a flat Riemannian metric.
\end{proof}

\begin{lemma}
\label{lemma:cases}
Let $(g,\varphi,f)$ be a non-trivial three-dimensional Heterotic soliton. Then, there exists a function $\phi \in C^{\infty}(M)$ such that $\varphi = \dd \phi$ and $f = c\, e^{\phi}$ for a non-zero constant $c\in \mathbb{R}^{\ast}$.
\end{lemma}
 
\begin{proof}
Let $(g,\varphi,f)$ be a non-trivial three-dimensional Heterotic soliton. We first prove that $f$ is nowhere vanishing. To do this, assume that $f(m) = 0$ for a given $m\in M$. For every smooth curve:
\begin{equation*}
\gamma \colon \cI \to M \, ,\quad t\mapsto \gamma_t
\end{equation*}

\noindent
where $\cI$ is an interval containing $0$ and $\gamma(0) = m$, the first equation in \eqref{eq:motionHetsugra3d2} implies:
\begin{equation*}
\partial_t(f\circ \gamma_t) = (\gamma^{\ast}\varphi)(\partial_t  (f\circ \gamma_t))
\end{equation*}

\noindent
which is a linear ordinary differential equation for $f\circ \gamma\colon \cI\to \mathbb{R}$. Since $f(m)=0$, the existence and uniqueness of solutions to the previous ordinary differential equation imply that $f\circ \gamma_t = 0$ for all $t\in \cI$. Since this holds for every such $\gamma\colon \cI\to M$ and $M$ is connected, we conclude that $f$ vanishes identically, which is not allowed since $(g,f)$ is by assumption a non-trivial Heterotic soliton, cf. Proposition \ref{prop:f0trivial}. Hence $f$ is nowhere vanishing, in which case the first equation in \eqref{eq:motionHetsugra3d2} is equivalent to $\varphi=\dd \phi$ and $f = c e^{\phi}$ for $\phi \in C^\infty(M)$ and for a non-zero constant $c\in\mathbb{R}^{\ast}$.
\end{proof}

\noindent
By the previous proposition, we will always assume from now on that $f= c e^{\phi}$ for any non-trivial three-dimensional Heterotic soliton $(g,\varphi,f)\in \Sol_{\kappa}(M)$. This allows us to denote three-dimensional Heterotic solitons simply as pairs $(g,f)\in \Sol_{\kappa}(M)$. In particular, we will sometimes refer to $f$ as the \emph{dilaton} of the three-dimensional Heterotic soliton $(g,f)\in \Sol_{\kappa}(M)$.

\begin{proposition}
Let $(g,f)\in \Sol_{\kappa}(M)$ be a non-trivial Heterotic soliton. Then, the scalar curvature $s_g$ of $g$ is strictly negative on some non-empty open set of $M$.
\end{proposition}
 
\begin{proof}
Let $(g,f)$ be a Heterotic soliton and consider the second equation in \eqref{eq:motionHetsugra3d2}, which can be equivalently rewritten as follows:
\begin{equation*}
2 f^2 s_g = - 6 f^2 \mathrm{Tr}_g(\mathrm{Hes}_g(\log(f))) + 4\vert \dd f\vert_g^2 -f^4 \, ,
\end{equation*}

\noindent
where $\mathrm{Hes}_g(\log(f))$ denotes the Hessian of $\log(f)$ with respect to $g$. If $f$ is constant then the previous equation reduces to $2 f^2 s_g = -f^4$, which immediately implies that $s_g$ is a strictly negative constant since $f$ cannot be zero. If $f$ is non-constant then evaluating the previous equation at an absolute minimum $m\in M$ of $f$ we obtain:
\begin{equation*}
2 f_m^2 (s_g)_m = - 6 f^2 \mathrm{Tr}_g(\mathrm{Hes}_g(\log(f)))_m -f^4_m 
\end{equation*}

\noindent
where the subscript $m$ denotes evaluation at $m\in M$. Since $m$ is an absolute minimum and $f^2$ is strictly positive, we have $\mathrm{Tr}_g(\mathrm{Hes}_g(\log(f)))_m\geq 0$ and we conclude.
\end{proof}

\begin{remark}
The previous proposition rules out the possibility of having three-dimensional Heterotic solitons with positive scalar curvature, which in dimension three is obstructed \cite{KazdanWarner}. In contrast, any function on $M$ that is negative on some open set is the scalar curvature of a metric on $M$.
\end{remark}

\noindent
Given the complexity of the Heterotic soliton system, it is natural to look for examples that satisfy natural curvature conditions. In the following result, we consider metrics whose Ricci tensor has constant principal curvatures, where the latter are defined as the eigenvalues of the Ricci endomorphism.

\begin{proposition}
\label{prop:fconstant}
A non-trivial three-dimensional Heterotic soliton $(g,f)\in \Sol_{\kappa}(M)$ with constant dilaton $f$ has constant principal Ricci curvatures $(\mu_1,\mu_2,\mu_3)$ which  satisfy one of the following conditions:
\begin{enumerate}[leftmargin=*]
\item $ \kappa f^2 = 1$  and $(\mu_1 = \mu_2 = -\frac{1}{2\kappa} , \mu_3 =  \frac{1}{2\kappa})$. 
	
\
	
\item $ \kappa f^2 = 2$  and $(\mu_1 = \mu_2 =0, \mu_3 =  -\frac{1}{\kappa})$. In particular, the universal cover of $M$ is isometric to either $\widetilde{\mathrm{Sl}}(2,\mathbb{R})$ or $\mathrm{E}(1,1)$ equipped with a left-invariant metric.
	
\

\item $ \kappa f^2 = 3$  and $(\mu_1 = \mu_2  = \mu_3 =  -\frac{1}{2\kappa})$. In particular $(M,g)$ is a hyperbolic three-manifold endowed with a metric of scalar curvature $-\frac{3}{2\kappa}$.
\end{enumerate}
\end{proposition}

\begin{proof}
Assume that $f$ is a constant, which by Proposition \ref{prop:f0trivial} must then be non-zero. With this assumption, Equations \eqref{eq:motionHetsugra3d1} and \eqref{eq:motionHetsugra3d2} reduce to:
\begin{eqnarray*}
-  \kappa \, \mathrm{Ric}^g \circ_g \mathrm{Ric}^g + (1 - \kappa f^2  )\mathrm{Ric}^{g}  + ( \kappa  \vert \mathrm{Ric}^g \vert_g^2  - \frac{1}{2} f^2)  g = 0\, , \qquad s_g=  -\frac{1}{2} f^2\, .
\end{eqnarray*}

\noindent
Taking the trace of the first equation and combining it with the second equation we find:
\begin{equation*}
2 \kappa \vert\mathrm{Ric}^g \vert_g^2 = 2 f^2 - \frac{\kappa f^4}{2}\,,
\end{equation*}

\noindent
which plugged back into the previous equations yields:
\begin{equation}\label{eq:e3d}
-  \kappa \, \mathrm{Ric}^g \circ_g \mathrm{Ric}^g + (1 - \kappa f^2  )\mathrm{Ric}^{g}  + \frac{1}{2} \left (f^2 - \frac{\kappa f^4}{2}\right)  g = 0\, , \qquad s_g=  -\frac{1}{2} f^2\, .
\end{equation}

\noindent
The discriminant of the previous equation can be verified to be one, and by solving it we obtain the cases listed in the statement of the Proposition, similarly to  \cite[Proposition 4.6 and Theorem 4.9]{Moroianu:2021kit}, whose proof involves solving the same algebraic equation for the Ricci tensor. 
\end{proof}

\begin{remark}
In case $(1)$ of the previous proposition, Reference \cite{Moroianu:2021kit} proves the existence of an associated Sasakian structure, possibly involving a different Riemannian metric, constructed out of the unit-norm eigenvector of $\mathrm{Ric}^g$ with positive eigenvalue. Very little is currently known about the basic properties of this auxiliary Sasakian structure.
\end{remark}

\noindent
Ideally, we would like to construct or, at least, prove the existence of three-dimensional Heterotic solitons $(g,f)$ with non-constant dilaton $f$. The following result excludes for this purpose the case of Heterotic solitons with Einstein metrics. 

\begin{theorem}
\label{thm:constantdilatonEinstein}
All Einstein three-dimensional Heterotic solitons have constant dilaton.
\end{theorem}

\begin{proof}
If $(M,g)$ is Einstein, so that $\mathrm{Ric}^g=\frac13 s_g g$, the Einstein equation \eqref{eq:motionHetsugra3d1} becomes:
\begin{equation}
    \label{eq:motionHetsugra3d1e}
 \left (\frac13 s_g-\frac\kappa6s_gf^2+\frac{\kappa}{18} s_g^2+\frac{\kappa}{4} \vert \dd f  \vert_g^2  - \frac{1}{2} f^2+\frac{\kappa}{8}f^4\right)  g  + \frac{\kappa }{4} \dd f \otimes \dd f + \nabla^{g}\varphi= 0\, .
\end{equation}
Taking the trace with respect to $g$ in this equation yields:
\begin{equation}
    \label{eq:motionHetsugra3d1et}
 s_g-\frac\kappa2 s_g f^2+\frac{\kappa}{6} s_g^2+\kappa \vert \dd f  \vert_g^2  - \frac{3}{2} f^2+\frac{3\kappa}{8}f^4 - \delta^{g}\varphi= 0\, .
\end{equation}
Using now the dilaton equations \eqref{eq:motionHetsugra3d2} we readily obtain $\vert \dd f\vert_g^2=\alpha(f)$, for some explicit rational function $\alpha$ given by:
\begin{equation*}
\alpha(t):=\frac{-\frac23 s_g+\frac\kappa2s_gt^2-\frac{\kappa}{6} s_g^2  +\frac{5}{3} t^2-\frac{3\kappa}{8}t^4}{\kappa+\frac2{3t^2}} = \frac{-2s_gt^2+\frac 32 \kappa s_gt^4-\frac 12\kappa s_g^2t^2 +5 t^4-\frac{9\kappa}{8}t^6}{3\kappa t^2+2}
\end{equation*}

\noindent
Applying then \eqref{eq:motionHetsugra3d1e} to (the metric dual of) $\dd f$ and using that:
\begin{equation*}
\nabla^g_{\dd f}\varphi=\nabla^g_{\dd f}\left(\frac{\dd f}f\right)=\frac12 \dd (\vert \dd f\vert_g^2)\frac1f-\frac{\vert \dd f\vert_g^2}{f^2}\dd f=\left(\frac{\alpha'(f)}{2f}-\frac{\alpha(f)}{f^2}\right)\dd f
\end{equation*}

\noindent
we obtain:
\begin{equation}
\label{eq:al}
\left (\frac13s_g-\frac\kappa6s_gf^2+\frac{\kappa}{18} s_g^2+\frac{\kappa}{2} \alpha(f)  - \frac{1}{2} f^2+\frac{\kappa}{8}f^4+\frac{\alpha'(f)}{2f}-\frac{\alpha(f)}{f^2}\right) \dd f= 0\, .
\end{equation}

\noindent
Since $\alpha(t)\sim_{t\to+\infty}-\frac38t^4$ and $\alpha'(t)\sim_{t\to+\infty} -\frac32 t^3$, we deduce from \eqref{eq:al} that $P(f)\dd f=0$, where $P$ is some (explicit) non-vanishing polynomial. This shows that the dilaton $f$ is constant.
\end{proof}

\noindent
If we weaken the condition of having constant principal Ricci curvatures in order to construct Heterotic solitons with non-constant dilaton we can rapidly encounter important obstructions, as illustrated by the following result.

\begin{proposition}
Let $(g,f)\in \Sol_{\kappa}(M)$ be a non-trivial three-dimensional Heterotic soliton such that $\dd f\neq 0$ and:
\begin{equation*}
\dd s_g = 0\, , \qquad \dd \vert\mathrm{Ric}^g\vert_g^2 = 0\, .
\end{equation*}

\noindent
Then $M$ is diffeomorphic to the sphere $S^3$.
\end{proposition}

\begin{proof}
Suppose $(g,f)$ is a Heterotic soliton with non-constant dilaton $f$. We evaluate Equations \eqref{eq:motionHetsugra3d1} and \eqref{eq:motionHetsugra3d2} at a critical point $m\in M$ of $f$ whose critical value we denote by $f_c$. We obtain:
\begin{eqnarray*}
& - \mathrm{Hes}(\log f)_m = -  \kappa \, \mathrm{Ric}^g_m \circ_g \mathrm{Ric}^g_m + (1 + \kappa (s_g)_m - \frac{\kappa}{2} f^2_c  )\mathrm{Ric}^{g}_m  + \left ( \kappa  \vert \mathrm{Ric}^g_m \vert_g^2-\frac{\kappa}{2} (s^2_g)_m- \frac{1}{2} f^2_c+\frac{\kappa}{8}f^4_c\right)  g\, , \\
& \mathrm{Tr}_g(\mathrm{Hes}(\log f)_m) = -\frac{(s_g)_m}{3}  -\frac{1}{6} f^2_c\, .
\end{eqnarray*}

\noindent
Taking the trace of the first equation and combining it with the second we obtain a quadratic equation for $f^2_c$ whose coefficients are constant due to the assumption $\dd s_g = 0$ and $\dd \vert\mathrm{Ric}^g\vert_g^2 = 0$. Hence, the function $f$ can have at most two critical values, implying that it is either constant, which is not allowed, or has exactly two critical points. The latter implies that $M$ is diffeomorphic to the sphere \cite{MilnorII,Rosen}. 
\end{proof}

\noindent
We note that the existence of three-dimensional Heterotic solitons with non-constant dilaton is currently an open problem.

%%%%%%%%%%%%%%%%%%%%%%%%%%%%%%%%%%%%%%%%%%%%%%%%%%%%%%%%%%%%%%%%%%%%%%%%%%%%

\subsection{Classification of three-dimensional strong Heterotic solitons}
\label{sec:3dstrongsolitons}

%%%%%%%%%%%%%%%%%%%%%%%%%%%%%%%%%%%%%%%%%%%%%%%%%%%%%%%%%%%%%%%%%%%%%%%%%%%%

In this subsection, we consider strong three-dimensional Heterotic solitons and provide their complete classification.

\begin{lemma}
\label{lemma:strongfconstant}
A strong three-dimensional Heterotic soliton $(g,f)\in \Sol_{\kappa}(M)$ has constant dilaton $f$.
\end{lemma}

\begin{proof}
Every strong Heterotic soliton satisfies by definition the strong condition \eqref{eq:strongcondition} and therefore also satisfies its skew-symmetric projection \eqref{eq:Lambda3strong}. Plugging $H= f\nu_g$ into equation \eqref{eq:Lambda3strong} and simplifying we obtain that it is equivalent to:
\begin{equation*}
\nabla^{g\ast} \nabla^g H + \nabla^g_{\varphi} H =0\, .
\end{equation*}

\noindent
Setting now $H = f \nu_g$ in the previous equation we obtain that it is equivalent to:
\begin{equation*}
f\Delta_g f + \vert\dd f\vert_g^2 = 0.
\end{equation*}

\noindent
Integrating over $M$, this implies $\dd f=0$.
\end{proof}
 
\noindent
The previous lemma reduces the study of three-dimensional strong Heterotic solitons to the study of the strong condition given at Equation \eqref{eq:strongcondition} in the three possible classes of three-dimensional Heterotic solitons listed in Proposition \ref{prop:fconstant}. Since $f$ is a non-zero constant, the system  \eqref{eq:motionHetsugra3d1}--\eqref{eq:motionHetsugra3d2} is equivalent to \eqref{eq:e3d} and the strong equation \eqref{eq:strongcondition} becomes:
\begin{equation}\label{eq:strong2}
\nabla^{g,H\ast}\mR^{g,H}=0\, .
\end{equation}
We easily compute $\mR^{g,H}_{v_1,v_2}=\mR^g+\frac14f^2v_1\wedge v_2$ for every tangent vectors $v_1,v_2$, therefore using the Bianchi identity $\nabla^{g\ast}\mR^{g}=\dd^{\nabla^g}\mathrm{Ric}^g$ we obtain after a straightforward computation that \eqref{eq:strong2} is equivalent to: 
\begin{equation}\label{eq:strong3}
  \dd^{\nabla^g}\mathrm{Ric}^g(v)+\frac{3f}2*(\mathrm{Ric}^g_0(v))  =0,\qquad \forall v\in TM\, ,
\end{equation}
where  $\mathrm{Ric}^g_0$ is the traceless Ricci tensor. If $(M,g)$ is Einstein, this equation is clearly satisfied, so the solitons in case (3) of Proposition \ref{prop:fconstant} satisfy automatically the strong condition. If $(g,f)$ is of type (1) or (2) of Proposition \ref{prop:fconstant}, the Ricci tensor has one double eigenvalue $\mu_1$ and one simple eigenvalue $\mu_3$. Let $\xi$ denote a unit vector field on $M$ satisfying $\mathrm{Ric}^g(\xi)=\mu_3\xi$ (here we might have to replace $M$ with a double cover in order for $\xi$ to be globally defined). We write: 
$$\mathrm{Ric}^g=\mu_1 g+(\mu_3-\mu_1)\xi\otimes\xi,\qquad \mathrm{Ric}^g_0=(\mu_3-\mu_1)\left(-\frac13g+\xi\otimes\xi\right),$$
so $\dd^{\nabla^g}\mathrm{Ric}^g=(\mu_3-\mu_1)(\dd\xi\otimes\xi+(e_i\wedge\xi)\otimes\nabla^g_{e_i}\xi)$. The strong equation \eqref{eq:strong3} becomes: 
$$g(v,\xi)\,\dd\xi +g(\nabla^g_{e_i}\xi,v)\,e_i\wedge\xi +\frac{3f}2*\left(-\frac13v+g(v,\xi)\,\xi \right)=0,\qquad\forall v\in TM\, .$$
For $v=\xi$ this gives $\dd\xi=-f*\xi$. Reinjecting in the above equation we obtain: 
$$g(\nabla^g_{e_i}\xi,v)\,e_i\wedge\xi =\frac{f}2*\left(v-g(v,\xi)\,\xi \right),\qquad\forall v\in TM\, .$$
Taking the interior product with $\xi$ in this equation and using again $\dd\xi=-f*\xi$ yields:
$$\frac f2*(v\wedge\xi)=g(\nabla^g_\xi\xi,v)\xi-g(\nabla^g_{e_i}\xi,v)\,e_i=\dd\xi(\xi,v) \xi-g(\nabla^g_v\xi,e_i)\,e_i-\dd\xi(e_i,v)\,e_i=-\nabla^g_v\xi-f*(\xi\wedge v)\, ,$$
showing that $\nabla_v\xi=-\frac12f*(\xi\wedge v)$ for every $v\in TM$. In particular, $\xi$ is Killing. The Bochner formula gives $2\mathrm{Ric}^g(\xi) =f^2\xi$ whence $2\mu_3=f^2=-2s_g=-2(2\mu_1+\mu_3)$, that is, $\mu_1+\mu_3=0$. We are thus in case (1) of Proposition \ref{prop:fconstant}. Reversing the order of the above calculations shows that, conversely, if the Ricci tensor of $(M,g)$ satisfies the condition (1) in Proposition \ref{prop:fconstant} and $\xi$ is a unit length vector field on $M$ satisfying $\nabla\xi=-\frac12f*\xi$, then $(M,g,f)$ is a strong three-dimensional soliton. It remains to characterize the manifolds satisfying these conditions.

\begin{proposition}
    Let $(M,g)$ be a complete simply connected three-dimensional Riemannian manifold such that there exists a positive constant $f$ and a unit vector field $\xi$ with $\nabla^g\xi=-\frac12f*\xi$ and $\mathrm{Ric}^g=\frac{1}2f^2(-g+2\xi\otimes\xi)$. Then $(M,g)$ is isometric to the Heisenberg group $\mathrm{H}_3$, endowed with a left-invariant metric. Conversely, for every left-invariant metric on $\mathrm{H}_3$ there exist $f$ and $\xi$ with the above properties.
\end{proposition}
\begin{proof}
    Consider the metric connection with torsion on $(M,g)$ defined by
    $$\bar\nabla_v:=\nabla^g_v-\frac12f*v+fg(v,\xi)*\xi\, .$$
    We have 
    $$\bar\nabla_v\xi:=\nabla^g_v\xi-\frac12f*v(\xi)=-\frac12f*\xi(v)-\frac12f*v(\xi)=0\, ,$$
    so $\xi$ is $\bar\nabla$-parallel. A straightforward but tedious computation then shows that $\bar\nabla$ is flat. Consequently, there exists a $\bar\nabla$-parallel oriented orthonormal frame $\xi,v,w$ on $TM$. We have:
    $$0=\bar\nabla_vw-\bar\nabla_wv=[v,w]-f*(v\wedge w)=[v,w]-f\xi\, ,$$
    $$0=\bar\nabla_v\xi-\bar\nabla_\xi v=[v,\xi]-\frac12f*v(\xi)+\frac12f*\xi(v)-f*\xi(v)=[v,\xi]\, ,$$
    and similarly $[w,\xi]=0$. This shows that the frame $(e_1,e_2,e_3):=(f\xi,v,w)$ satisfies the commutator relations 
    $$[e_1,e_3]=[e_2,e_3]=0,\ [e_2,e_3]=e_1$$
    of the Heisenberg Lie algebra $\mathfrak{h}_3$, and thus $(M,g)$ is isometric to $\mathrm{H}_3$ endowed with the left-invariant metric determined by the scalar product $f^2e^1\otimes e^1+e^2\otimes e^2+e^3\otimes e^3$.

    Conversely, if $f$ is a positive constant and $(M,g)$ is the Heisenberg group endowed with the above left-invariant metric, the standard formulas of Riemannian Lie groups show that the unit left-invariant vector field $\xi$ induced by $\frac1fe_1$ satisfies $\nabla^g\xi=-\frac12f*\xi$ and that $\mathrm{Ric}^g=\frac{1}2f^2(-g+2\xi\otimes\xi)$.
\end{proof}

\noindent
Altogether, the previous discussion results in the following classification of compact strong Heterotic solitons in three dimensions.

\begin{theorem}
\label{thm:strongHeterotic3d}
Every compact three-dimensional strong Heterotic soliton has constant dilaton and is either isometric to a compact quotient of the Heisenberg group equipped with a left-invariant metric or homothetic to a compact hyperbolic three-manifold.  
\end{theorem}

%%%%%%%%%%%%%%%%%%%%%%%%%%%%%%%%%%%%%%%%%%%%%%%%%%%%%%%%%%%%%%%%%%%%%%%%%%%%

\section{Rigidity of Einstein Heterotic solitons}
\label{sec:rigidity}

%%%%%%%%%%%%%%%%%%%%%%%%%%%%%%%%%%%%%%%%%%%%%%%%%%%%%%%%%%%%%%%%%%%%%%%%%%%%

All known Heterotic solitons on a compact three-manifold have constant principal Ricci curvatures and constant dilaton $f$. In this section, we study the local moduli of three-dimensional Heterotic solitons around an Einstein Heterotic soliton, which must be then either flat or of negative constant sectional curvature, in order to evaluate the possibility of deforming them to construct compact three-dimensional Heterotic solitons with non-constant dilaton. For definiteness, we will exclusively focus on the deformation of non-trivial Heterotic solitons with constant dilaton. The deformation problem of trivial Heterotic solitons can be easily studied by directly inspecting the Heterotic soliton system, from which the rigidity of trivial Heterotic solitons follows.

In three dimensions the Bianchi identity is automatically satisfied and, after solving the Maxwell equation as prescribed in Lemma \ref{lemma:cases} the maps introduced in \eqref{eq:eqsmap} reduce to:
\begin{eqnarray*}
\cE = (\cE_{\mathrm{E}} , \cE_{\mathrm{D}}) \colon \Conf(M) \to \Gamma(T^{\ast}M\otimes T^{\ast}M)\times \cC^{\infty}(M)
\end{eqnarray*}

\noindent
namely to the maps defined by the Einstein and dilaton equations. The diffeomorphism group $\Diff(M)$ acts naturally on $\Conf(M)$. This action is smooth in the Fréchet category when we consider $\Conf(M)$ as a Fréchet manifold and $\Diff(M)$ as a Fréchet-Lie group. Furthermore, the map $\cE$ is diffeomorphism-equivariant, namely:
\begin{equation*}
\cE (u^{\ast}g,f\circ u) =(u^{\ast}\cE_{\mathrm{E}}(g,f) , \cE_{\mathrm{D}}(g,f)\circ u)
\end{equation*}

\noindent
for every $(g,f)\in \Conf(M)$ and $u\in \Diff(M)$. Hence, the action maps Heterotic solitons to Heterotic solitons. The moduli space of three-dimensional Heterotic solitons on a closed three-manifold $M$ is then defined as follows:
\begin{equation*}
\mathfrak{M}(M) := \cE^{-1}(0)/\Diff(M)\, .
\end{equation*}

\noindent
It is equipped with the quotient topology of the subspace topology induced by the $\cC^{\infty}$ topology of $\Conf(M)$ corresponding to its Fréchet manifold structure. The main result that we need to understand the local structure of $\mathfrak{M}(M)$ is a \emph{slice theorem} for the action of $\Diff(M)$ on $\Conf(M)$.  

%%%%%%%%%%%%%%%%%%%%%%%%%%%%%%%%%%%%%%%%%%%%%%%%%%%%%%%%%%%%%%%%%%%%%%%%%%%%

\subsection{Local slice in configuration space}
\label{sec:local slice}

%%%%%%%%%%%%%%%%%%%%%%%%%%%%%%%%%%%%%%%%%%%%%%%%%%%%%%%%%%%%%%%%%%%%%%%%%%%%

The fact that we are studying Heterotic solitons in three dimensions together with the simplification given by Lemma \ref{lemma:cases}, implies that the configuration space of the system reduces in this case to that of  the classical Ricci solitons, namely pairs $(g,f)$ as described above, although the equations that define our system are remarkably more complicated than the standard Ricci soliton system. Furthermore, for us $f$ is a variable of the system whereas in the theory of Ricci solitons $f$ is usually considered as being completely determined as an eigenvalue of a Schrödinger type of operator associated to the system. The moduli space of standard Ricci solitons together with a slice theorem for the corresponding action of the diffeomorphism group is proved in \cite{PodestaSpiro} by considering Sobolev completions of the  spaces involved.  Here we adopt a direct approach and we consider the smooth right action $\Psi\colon \Conf(M)\times \Diff(M)\to \Conf(M)$ of $\Diff(M)$ on $\Conf(M)$ via pull-back. For every $(g,f)\in \Conf(M)$ define:
\begin{equation*}
\Psi_{g,f} \colon \Diff(M) \to \Conf(M)\, , \qquad u\mapsto (u^{\ast}g,f\circ u)
\end{equation*}

\noindent 
to be the orbit map associated with $(g,f)$. In particular, $\cO_{g,f} := \Im{\rm m}(\Psi_{g,f})$ is the orbit of the action $\Psi$ passing through $(g,f)$.  The differential of  $\Psi_{g,f} \colon \Diff(M) \to \Conf(M)$ at the identity $e\in \Diff(M)$ reads:
\begin{eqnarray*}
\dd_e\Psi_{g,f} \colon \mathfrak{X}(M) \to T_{g,f}\Conf(M) \, , \qquad v \mapsto \mathcal{L}_v g \oplus \dd f(v)
\end{eqnarray*}
\noindent
where the symbol $\mathcal{L}$ denotes the Lie derivative. Its $L^2$ adjoint is given by:
\begin{equation*}
\dd^{\ast}_e\Psi_{g,f} \colon T_{g,f}\Conf(M)  \to \mathfrak{X}(M) \, , \qquad h\oplus \sigma \mapsto 2(\nabla^{g\ast} h)^{\sharp_g} + \sigma \dd f^{\sharp_g}
\end{equation*}

\begin{lemma}
We have a $L^2$ orthogonal decomposition:
\begin{equation*}
T_{g,f}\Conf(M) = \Im{\rm m}(\dd_e\Psi_{g,f}) \oplus \Ker(\dd_e\Psi^{\ast}_{g,f})
\end{equation*}
\noindent
in terms of closed subspaces of $T_{g,f}\Conf(M)$.
\end{lemma}

\begin{proof}
The result follows from the fact that the symbol of $\dd_e\Psi_{g,f} \colon \mathfrak{X}(M) \to T_{g,f}\Conf(M) $ is injective.
\end{proof}

\noindent
Therefore, if $(g,f)$ is a three-dimensional Einstein Heterotic soliton with constant dilaton, then the projection of $\dd_e\Psi_{g,f} \colon \mathfrak{X}(M) \to T_{g,f}\Conf(M)$ to $C^{\infty}(M)$ is zero and the slice for the action of $\Diff(M)$  around $(g,f)$ can be characterized as follows.

\begin{proposition}
Let $(g,f)$ be a three-dimensional Einstein Heterotic soliton with constant dilaton. Any smooth submanifold  $\cS_{g.f}\subset\Conf(M)$ of the form:
\begin{equation*}
\cS_{g,f} = \cS_g\times C^{\infty}(M)
\end{equation*}
\noindent
where $\cS_g\subset \Met(M)$ is a slice for the action of $\Diff(M)$ on $\Met(M)$ is a slice for the action of $\Diff(M)$ on $\Conf(M)$. In particular:
\begin{equation*}
T_{g,f}\cS_{g,f} = \Ker(\nabla^{g\ast}\colon \Gamma(T^{\ast}M\odot T^{\ast}M)\to \Omega^1(M))\oplus C^{\infty}(M)
\end{equation*}
\noindent
is the tangent bundle of $\cS_{g,f}$ at $(g,f)$.
\end{proposition}

\begin{remark}
Thanks to recent developments in the theory of infinite dimensional groups and actions \cite{DiezRudolph} we do not need to work in any Sobolev completion of $\Conf(M)$ and we can remain in the category of Fréchet spaces and Fréchet Lie groups of smooth sections.
\end{remark}

%%%%%%%%%%%%%%%%%%%%%%%%%%%%%%%%%%%%%%%%%%%%%%%%%%%%%%%%%%%%%%%%%%%%%%%%%%%%

\subsection{Essential deformations}
\label{sec:essentialdeformations}

%%%%%%%%%%%%%%%%%%%%%%%%%%%%%%%%%%%%%%%%%%%%%%%%%%%%%%%%%%%%%%%%%%%%%%%%%%%%

Let $(g,f)$ be a non-trivial three-dimensional Einstein Heterotic soliton with constant dilaton. Recall that in this case, the Heterotic soliton system implies:
\begin{equation*}
\mathrm{Ric}^g= - \frac{f^2}{6} g\, , \qquad \kappa f^2 = 3\, .
\end{equation*}
\noindent
We will tacitly use these equations in the following. The existence of a natural slice for the action of the diffeomorphism group around every point $(g,f)\in \Conf(M)$ immediately implies the following inclusion of vector spaces:
\begin{equation*}
T_{g,f}\mathfrak{M}(M) \subset \Ker(\dd_{g,f} \cE)\cap \Ker(\dd_e\Psi^{\ast}_{g,f})
\end{equation*}
\noindent
where $T_{g,f}\mathfrak{M}(M)$ the tangent space of $\mathfrak{M}(M)$ at the class $[g,f]\in \mathfrak{M}$ determined by $(g,f)$ and:
\begin{equation*}
\dd_{g,f} \cE \colon T_{g,f} \Conf(M) \to T_{g,f} \Conf(M)
\end{equation*}

\noindent
is the differential of $\cE$ at $(g,f)$. This motivates the following definition.

\begin{definition}
Let $(g,f)$ be a non-trivial three-dimensional Einstein Heterotic soliton with constant dilaton. The vector space $\mathbb{E}_{g,f}$ of \emph{essential deformations} of $(g,f)$ is:
\begin{eqnarray*}
\mathbb{E}_{g,f} = \Ker(\dd_{g,f} \cE)\cap \Ker(\dd_e\Psi^{\ast}_{g,f})
\end{eqnarray*}
\noindent
If $\mathbb{E}_{g,f} = 0$ then $(g,f)$ is rigid.
\end{definition}

\begin{remark}
By the existence of a slice around $(g,f)$, the previous notion of rigidity corresponds to the intuitive one, namely, if $(g,f)$ is rigid then there exists a neighborhood of $(g,f)$ in $\cE^{-1}(0)$ in which every element is diffeomorphic to $(g,f)$.
\end{remark}

\begin{lemma}
\label{lemma:laplacians}
Let $(g,f)$ be a non-trivial three-dimensional compact Einstein Heterotic soliton with constant dilaton. Then, the following equation holds:
\begin{equation*}
\Delta_g\mathrm{Tr}_g(h) = f^2 \mathrm{Tr}_g(h)\, , \qquad  \sigma = \frac{7f}{12}\mathrm{Tr}_g(h)
\end{equation*}
\noindent
for every $(h,\sigma)\in \mathbb{E}_{g,f}$.
\end{lemma}

\begin{proof}
Let $(g,f)$ be a three-dimensional Einstein Heterotic soliton with a constant dilaton and fix a pair $(h,\sigma)\in \Ker(\dd_e\Psi^{\ast}_{g,f})$. We have:
\begin{equation*}
(\dd_g\mathrm{Ric})(h) = \frac{1}{2} \nabla^{g\ast}\nabla^{g}h - \frac{f^2}{4} h  + \frac{f^2}{12} \mathrm{Tr}_g(h) g - \frac{1}{2} \nabla^g\dd \mathrm{Tr}_g(h)\, , \quad (\dd_g s) (h) = \Delta_g \mathrm{Tr}_g(h) + \frac{f^2}{6} \mathrm{Tr}_g(h)
\end{equation*}

\noindent
which follow from the standard formulas for the linearization of the Ricci and scalar curvatures \cite{Besse} and Equation \eqref{eq:Riemann3d} in Appendix \ref{app:curvature3d}. Note that the latter implies in particular:
\begin{equation*}
\mR^g_{v_1,v_2} = \frac{f^2}{12} v_1\wedge v_2 \qquad v_1,v_2 \in TM\, .
\end{equation*}

\noindent
Using the previous formulae we compute:
\begin{eqnarray*}
& (\dd_g \vert \mathrm{Ric} \vert_g^2) (h) = - 2 g(h, \mathrm{Ric}^g\circ_g \mathrm{Ric}^g) + 2 g(\dd_g\mathrm{Ric}(h), \mathrm{Ric}^g) = -\frac{f^4}{18} \mathrm{Tr}_g(h)\\
& -\frac{f^2}{6} \mathrm{Tr}_g(\nabla^{g\ast}\nabla^{g}h - \frac{f^2}{2} h  + \frac{f^2}{6} \mathrm{Tr}_g(h) g - \nabla^g\dd \mathrm{Tr}_g(h)) = -\frac{f^2}{3} (\Delta_g \mathrm{Tr}_g(h) + \frac{f^2}{6}  \mathrm{Tr}_g(h))\, .
\end{eqnarray*}
\noindent
In particular:
\begin{equation*}
(\dd_g \vert \mathrm{Ric} \vert_g^2) (h) = -\frac{f^2}{3} (\dd_g s)(h)\, .
\end{equation*}

\noindent
Assuming now that $(h,\sigma)\in \mathbb{E}_{g,f}$ holds, the differential of the trace of the Einstein equation \eqref{eq:motionHetsugra3d1} at $(g,f)$ gives:
\begin{eqnarray*}
 & 2\kappa \, \dd_g\vert \mathrm{Ric}^g \vert_g^2(h) - \frac{3}{2} \dd_g s(h) +\frac{3}{2} f \sigma   -2 \dd_g s(h) + 3 (\frac{3}{2}f \sigma  +\frac{3}{2} \dd_g s(h)  - f \sigma)  -\frac{1}{f} \Delta_{g}\sigma\\
 & = 3 f \sigma -\dd_g s(h)  -\frac{1}{f} \Delta_{g}\sigma = 0\, .
\end{eqnarray*}

\noindent
On the other hand, the differential of equation \eqref{eq:motionHetsugra3d2} at $(g,f)$ immediately gives:
\begin{equation*}
\dd_g s(h) = \frac{3}{f}\Delta_g \sigma - f\sigma \, .
\end{equation*}
\noindent
Combining the previous two equations we obtain:
\begin{equation*}
\Delta_g \sigma = f^2 \sigma
\end{equation*}
\noindent
and a quick computation reveals that:
\begin{equation*}
\Delta_g (\mathrm{Tr}_g(h) - \frac{12}{7f} \sigma) = -\frac{f^2}{6} (\mathrm{Tr}_g(h) - \frac{12}{7f} \sigma)
\end{equation*}
\noindent
whence:
\begin{equation*}
\mathrm{Tr}_g(h) = \frac{12}{7f} \sigma
\end{equation*}
\noindent
and we conclude.
\end{proof}

\begin{remark}
By the previous lemma, if $(h,\sigma)\in \mathbb{E}_{g,f}$ then:
\begin{equation*}
(\dd_g s) (h) = \frac{7 f^2}{6} \mathrm{Tr}_g(h)\, , \qquad (\dd_g \vert \mathrm{Ric} \vert_g^2) (h) = -\frac{7 f^4}{18}\mathrm{Tr}_g(h)\, .
\end{equation*}
\noindent
We will tacitly use these identities in the following.
\end{remark}

\begin{lemma}
\label{lemma:nablaastnablah}
Let $(g,f)$ be a non-trivial three-dimensional compact Einstein Heterotic soliton with constant dilaton. Then, the following equations hold:  
\begin{equation*}
\nabla^{g\ast}\nabla^g h = \frac{f^2}{6} h + f^2 \mathrm{Tr}_g(h) g + \frac{13}{6} \nabla^g \dd \mathrm{Tr}_g(h)\, ,  \quad \Delta_g \mathrm{Tr}_g(h) = f^2 \mathrm{Tr}_g (h)
\end{equation*}

\noindent
for every $(h,\sigma)\in \mathbb{E}_{g,f}$.
\end{lemma}

\begin{proof}
Let $(g,f)$ be a non-trivial three-dimensional Einstein Heterotic soliton with constant dilaton. Differentiating equation \eqref{eq:motionHetsugra3d1} at $(g,f)$, manipulating and repeatedly using Lemma \ref{lemma:laplacians} we obtain:
\begin{eqnarray*}
-\dd_g\mathrm{Ric}(h) - \frac{f^2}{6}h + f\sigma g  + \frac{1}{f} \nabla^g \dd \sigma = -\frac{1}{2} \nabla^{g\ast}\nabla^{g}h + \frac{f^2}{12} h  + \frac{f^2}{2} \mathrm{Tr}_g(h) g  + \frac{13}{12} \nabla^g \dd\mathrm{Tr}_g(h)
\end{eqnarray*}
\noindent
and we conclude.
\end{proof}

\begin{theorem}
\label{thm:rigidity}
Let $(g,f)$ be a non-trivial three-dimensional compact Einstein Heterotic soliton with constant dilaton. Then $\mathbb{E}_{g,f}=0$.
\end{theorem}

\begin{proof}
Fix a non-trivial three-dimensional compact Einstein Heterotic soliton $(g,f)$ with constant dilaton and a pair $(h,\sigma)\in \mathbb{E}_{g,f}$.  Denote by  $\dd^g \colon \Omega^1(M,TM) \to \Omega^2(M,TM)$ the exterior covariant derivative defined by the Levi-Civita connection of $(M,g)$ on the bundle of one-forms taking values in $TM$. Denote by $\dd^{g\ast} \colon \Omega^2(M,TM) \to \Omega^1(M,TM)$ its formal adjoint. Our starting point is the celebrated Weitzenböck formula:
\begin{equation*}
(\dd^g\dd^{g\ast} + \dd^{g\ast}\dd^{g})(\alpha\otimes v) = \nabla^{g\ast}\nabla^g(\alpha\otimes v) + q^g(\alpha\otimes v)
\end{equation*}
\noindent
where $\alpha \in \Omega^1(M)$ , $v\in \mathfrak{X}(M)$ and $q_g$ is the linear operator given explicitly by:
\begin{equation*}
q^g(\alpha\otimes v) = \mathrm{Ric}^g(\alpha)\otimes v + e_i\otimes \mR^g_{\alpha,e_i}v
\end{equation*}
\noindent
in terms of an orthonormal basis $(e_i)$. We compute:
\begin{eqnarray*}
q^g(h) = q^g(h(e_k)\otimes e_k) = \mathrm{Ric}^g(h(e_k)\otimes e_k + e_i\otimes \mR^g_{e_k,e_i} h(e_k) = -\frac{f^2}{4} h_0
\end{eqnarray*}
\noindent
where $h_0$ is the trace-free part of $h$ and we have used that:
\begin{equation*}
\mathrm{Ric}^g = -\frac{f^2}{6} g\, , \qquad \mR^g_{v_1,v_2} = \frac{f^2}{12} v_1\wedge v_2\, , \quad v_1,v_2\in TM\, .
\end{equation*}
\noindent
We obtain:
\begin{equation*}
\dd^{g\ast}\dd^g h = \nabla^{g\ast}\nabla^g h - \frac{f^2}{4} h_0 = -\frac{f^2}{12} h + \frac{13 f^2}{12} \mathrm{Tr}_g(h) g + \frac{13}{6} \nabla^g \dd \mathrm{Tr}_g(h)
\end{equation*}
\noindent
upon use of Lemma \ref{lemma:nablaastnablah}. We apply now $\dd^{g\ast}\colon \Omega^1(M,TM)\to \mathfrak{X}(M)$ to both sides of the previous equation. Observing that the result of applying $\dd^{g\ast}$ to each monomial in the previous equation is a constant times $\mathrm{Tr}_g(h)$ and that the combination of all terms does not vanish we conclude:
\begin{equation*}
\dd \mathrm{Tr}_g(h) = 0
\end{equation*}
\noindent
which in turn implies $\mathrm{Tr}_g(h) = 0$ since $\Delta_g \mathrm{Tr}_g(h) = f^2 \mathrm{Tr}_g(h)$ and $f\neq0$ by assumption. Hence, by Lemma \ref{lemma:laplacians} we also have $\sigma = 0$ and by Lemma \ref{lemma:nablaastnablah} we finally end up with:
\begin{equation*}
\dd^{g\ast}\dd^g h = -\frac{f^2}{12} h\, .
\end{equation*}
\noindent
This immediately implies $h=0$ since the differential operator of the left hand side of the equation is positive.
\end{proof}

\noindent
The existence of a slice for the action of the diffeomorphism group around $(g,f)$ together with the previous theorem immediately implies the following result.

\begin{corollary}
Three-dimensional compact Einstein Heterotic solitons with constant dilaton are rigid.
\end{corollary}

\begin{remark}
We are not aware of any rigidity result for a compact solution of a supergravity theory, especially in the non-supersymmetric case. 
\end{remark}

%%%%%%%%%%%%%%%%%%%%%%%%%%%%%%%%%%%%%%%%%%%%%%%%%%%%%%%%%%%%%%%%%%%%%%%%%%%%
%%%%%%%%%%%%%%%%%%%%%%%%%%%%%%%%%%%%%%%%%%%%%%%%%%%%%%%%%%%%%%%%%%%%%%%%%%%%
%%%%%%%%%%%%%%%%%%%%%%%%%%%%%%%%%%%%%%%%%%%%%%%%%%%%%%%%%%%%%%%%%%%%%%%%%%%%
%%%%%%%%%%%%%%%%%%%%%%%%%%%%%%%%%%%%%%%%%%%%%%%%%%%%%%%%%%%%%%%%%%%%%%%%%%%%

\appendix

%%%%%%%%%%%%%%%%%%%%%%%%%%%%%%%%%%%%%%%%%%%%%%%%%%%%%%%%%%%%%%%%%%%%%%%%%%%%
%%%%%%%%%%%%%%%%%%%%%%%%%%%%%%%%%%%%%%%%%%%%%%%%%%%%%%%%%%%%%%%%%%%%%%%%%%%%
%%%%%%%%%%%%%%%%%%%%%%%%%%%%%%%%%%%%%%%%%%%%%%%%%%%%%%%%%%%%%%%%%%%%%%%%%%%%
%%%%%%%%%%%%%%%%%%%%%%%%%%%%%%%%%%%%%%%%%%%%%%%%%%%%%%%%%%%%%%%%%%%%%%%%%%%%
 
%%%%%%%%%%%%%%%%%%%%%%%%%%%%%%%%%%%%%%%%%%%%%%%%%%%%%%%%%%%%%%%%%%%%%%%%%%%%

\section{Riemannian curvature in three dimensions}
\label{app:curvature3d}

%%%%%%%%%%%%%%%%%%%%%%%%%%%%%%%%%%%%%%%%%%%%%%%%%%%%%%%%%%%%%%%%%%%%%%%%%%%% 

Let $(M,g)$ be a Riemannian three-manifold. The Riemannian tensor $\mR^g$ of $g$ can be written as follows in terms of its Ricci tensor $\mathrm{Ric}^g$:
\begin{equation}
\label{eq:Riemann3d}
\mR^g_{v_1,v_2} = \frac{s_g}{2} v_1\wedge v_2 + v_2\wedge \mathrm{Ric}^g(v_1) + \mathrm{Ric}^g(v_2)\wedge v_1\, , \qquad \forall v_1,v_2 \in TM\, ,
\end{equation}

\noindent
where $s_g$ is the scalar curvature of $g$. In particular, using the previous formula it is easy to show that the contraction $\mR^g\circ_g \mR^g$, is given by:
\begin{equation}
\label{eq:RgoRg3d}
\mR^g\circ_g \mR^g =  -  \mathrm{Ric}^g\circ \mathrm{Ric}^g +  s_g \mathrm{Ric}^g + (\vert \mathrm{Ric}^g\vert^2_g - \frac{s_g^2}{2}) g\, .
\end{equation}

\noindent
where we have defined:
\begin{equation*}
\mathrm{Ric}^g\circ_g \mathrm{Ric}^g(v_1,v_2) := g(\mathrm{Ric}^g(v_1),\mathrm{Ric}^g(v_2))\, , \qquad \forall v_1, v_2\in TM\, .
\end{equation*}

\noindent
In particular, the norm of $\mathrm{R}^g$ is given by:
\begin{equation*}
\vert \mR^g \vert_g^2 = \frac{1}{2} \mathrm{Tr}_g(\mR^g\circ_g \mR^g)  =  \vert \mathrm{Ric}^g \vert_g^2 - \frac{1}{4} s_g^2	\, .
\end{equation*}

\noindent
Using the previous formulae, we can obtain an explicit expression for $\mR^{g,H} \circ_g\mR^{g,H}$ in terms of standard curvature tensors of $g$ and derivatives of $f = \ast_g H \in C^{\infty}(M)$.  

\begin{lemma}
\label{lemma:vg}
Let $g$ be a Riemannian metric on $M$ and $H\in \Omega^3(M)$. The following equation holds:
\begin{eqnarray}
\label{eq:lemmavg}
& \mR^{g,H} \circ_g \mR^{g,H}  =- \mathrm{Ric}^g \circ \mathrm{Ric}^g + ( s_g - \frac{f^2}{2})  \mathrm{Ric}^g+( \vert \mathrm{Ric}^g \vert_g^2-\frac{s_g^2}{2} + \frac{1}{4}\vert \dd f \vert_g^2 + \frac{1}{8} f^4) g \nonumber \\ & + \frac12[\ast_g\dd f , \mathrm{Ric}^g]  +\frac{1}{4} \dd f \otimes \dd f \, ,
\end{eqnarray}
	
\noindent
where we have set $f  := \ast_g H$.
\end{lemma}

\begin{remark}
Here $[\ast_g\dd f , \mathrm{Ric}^g] \in \Gamma(\End(TM))$ denotes the standard commutator of endomorphisms, understanding both $\ast_g\dd f$ and $\mathrm{Ric}^g$ as endomorphisms of $TM$ by means of the musical isomorphisms determined by $g$.
\end{remark}

\begin{proof}
Since $H=f \nu_g$, we have:
\begin{equation*}
\nabla_{v_1}^{g,H} v_2=\nabla_{v_1}^g v_2-\frac{f}{2} \ast_g (v_1^\flat \wedge v_2^\flat)^{\sharp_g} \, , \qquad v_1,v_2 \in \mathfrak{X}(M)\, .
\end{equation*}
	
\noindent
Using the previous formula we compute:
\begin{equation*}
\mR^{g,H}_{v_1,v_2}  =\mR^g_{v_1,v_2} -\frac{1}{2}   (v_1(f) \ast_g v_2-v_2(f) \ast_g v_1)    + \frac{f^2}{4}v_1\wedge v_2 \, , \qquad v_1,v_2 \in \mathfrak{X}(M) \, .
\end{equation*}
	
\noindent
Substituting this expression into $\mR^{g,H} \circ\mR^{g,H}$ and expanding we obtain \eqref{eq:lemmavg} and we conclude. 
\end{proof} 

\noindent
Using the previous lemma we can compute one-half of the trace of $\mR^{g,H} \circ_g \mR^{g,H}$, which corresponds to twice the norm of $\mR^{g,H}$, obtaining:
\begin{equation*}
\vert \mR^{g,H}\vert^2_{g}  = \frac{1}{2} \mathrm{Tr}_g (\mR^{g,H} \circ_g \mR^{g,H})= \vert \mathrm{Ric}^g\vert^2_g - \frac{1}{4} s_g^2 - \frac{1}{4} f^2 s_g + \frac12 \vert\dd f\vert^2_g + \frac{3}{16} f^4\, .
\end{equation*}

\noindent
This formula is tacitly used in several computations throughout the paper. 

%\newpage
%\renewcommand{\leftmark}{\MakeUppercase{Bibliography}}
\phantomsection
\bibliographystyle{JHEP}
%\bibliographystyle{plain}
%\newpage  

% % % % % % % % % % % % % % % % % % % % % % % % % % % % % % % % % % % % % % 
% % % % % % % % % % % % % % % % % % % % % % % % % % % % % % % % % % % % % %

\end{document}